\documentclass[3p,10pt,a4paper,twoside,fleqn,sort&compress]{amsart}
\usepackage{graphicx, amssymb, amsfonts}
\newtheorem{theorem}{Theorem}[section]

\newtheorem{corollary}[theorem]{Corollary}

\newtheorem{definition}[theorem]{Definition}

\newtheorem{lemma}[theorem]{Lemma}

\newtheorem{remark}[theorem]{Remark}

\begin{document}

\title{On compressions and generalized spectra of operators over $C^{*}$-algebras }

\maketitle
\begin{center}
	\author{Stefan Ivkovi\'{c}}
\end{center}
\begin{abstract}
	In the classical operator theory, there are several versions of spectra, related to special classes of operators (Fredholm, semi-Fredholm, upper/lower semi-Fredholm, etc.). We generalize these notions for adjointable operators on Hilbert $C^*$-modules replacing scalars by the center of the algebra, and show that most relations between these spectra are still true for these generalized versions. The relation between these spectra of an operator and those of its compressions is also transferred to the case of Hilbert $C^*$-modules. 
\end{abstract}

\section{Introduction }
The Fredholm and semi-Fredholm theory on Hilbert and Banach spaces started by studying the certain integral equations which was done in the pioneering work by Fredholm in 1903 in \cite{F}. After that the abstract theory of Fredholm and semi-Fredholm operators on Banach spaces was during the time developed in numerous papers.\\
In this theory one studies for instance the properties and the relationship between certain classes of operators on a Banach space $X$ among which the most elementary are $\Phi(X), \Phi_{+}(X), \Phi_{-}(X),$ the classes of all Fredholm, upper semi-Fredholm and lower semi-Fredholm operators on X, respectively. For the applications in this paper, we also need to recall the definition of the following subclasses:
$$\Phi_{o}(X)=\lbrace A \in \Phi(X) \mid \mathrm{ index } A=0 \rbrace ,$$
$$\Phi_{+}^{-}(X)=\lbrace A \in \Phi_{+}(X)\mid \mathrm{ index } A<0 \rbrace,  $$
$$\Phi_{-}^{+}(X)=\lbrace A \in \Phi_{-}(X)\mid \mathrm{ index } A>0 \rbrace . $$
Some aspects of spectral semi-Fredholm theory were investigated in \cite{MV}, \cite{ZE}. In \cite{ZE} Zemanek considered compressions of bounded operators on Banach spaces and he gave a description of the Weyl spectrum of a bounded operator A in terms of the intersection of ordinary spectra of all compressions of A. More precisely he obtained:\\
%$$ \partial \sigma_{e w}^{\mathcal{A}} (F)  \subseteq  \partial \sigma_{e \tilde b}^{\mathcal{A}} (F) \subseteq  \partial \sigma_{eb}^{\mathcal{A}} (F)  $$  
$\sigma_{e w}({A}) = \cap \lbrace \sigma (A_{p}) \mid A_{p}$ is a compression of $A \rbrace $ where $A$ is a bounded operator on a Banach space $X$  and  $\sigma_{e w}(A)=\lbrace \lambda \in \mathbb{C} \mid (A-\lambda) \notin \Phi_{o}(X) \rbrace .$ Similarly he obtained \\
\begin{center}
	$\sigma_{e a}({A}) = \cap \lbrace \sigma_{a} (A_{p}) \mid A_{p}$ is a compression of $A \rbrace ,$
\end{center}
\begin{center}
	$\sigma_{e d}({A}) = \cap \lbrace \sigma_{d} (A_{p}) \mid A_{p}$ is a compression of $A \rbrace ,$
\end{center}
where 
\begin{center}
	$\sigma_{e a}({A}) = \lbrace \lambda \in \mathbb{C} \mid (A-\lambda) \notin \Phi_{+}^{-} (X) \rbrace ,$ 
\end{center}
\begin{center}
	$\sigma_{e d}({A}) = \lbrace \lambda \in \mathbb{C} \mid (A-\lambda) \notin \Phi_{-}^{+} (X) \rbrace ,$ 
\end{center}
%where as $\sigma_{a}({A_{p}}), \sigma_{d}({A_{p}}) $ denote approximate point spectrum and approximate defect spectrum of $A_{p},$ respectively.\\
\begin{center}
	$\sigma_{a}(A_{p})=\lbrace \lambda \in \mathbb{C} \mid (A_{p}-\lambda)$ is not bounded below $\rbrace,$
\end{center}
\begin{center}
	$\sigma_{d}(A_{p})=\lbrace \lambda \in \mathbb{C} \mid (A_{p}-\lambda)$ is not surjective $\rbrace$.
\end{center}

In \cite{MV} Mili\v{c}i\'{c} and Veseli\'{c} have considered boundaries of several kinds of Fredholm spectra of bounded operators on a Banach space $X$ and they have obtained certain chains of inclusions of the boundaries of these Fredholm spectra such as \\
$$\partial \sigma_{e w}(A) \subset \partial \sigma_{e f}(A) \subset 
\begin{array}{l}
\partial \sigma_{e \beta}(A) \\
\partial \sigma_{e \alpha}(A)\\
\end{array}
\subset \partial \sigma_{e k}(A)$$
for $A \in B(X)$ where 
\begin{center}
	$\sigma_{e f}(A) = \lbrace \lambda \in \mathbb{C} \mid (A-\lambda) \notin \Phi (X) \rbrace ,$ 
\end{center}
\begin{center}
	$\sigma_{e \alpha}(A) = \lbrace \lambda \in \mathbb{C} \mid (A-\lambda) \notin \Phi_{+} (X) \rbrace ,$ 
\end{center}
\begin{center}
	$\sigma_{e \beta}(A) = \lbrace \lambda \in \mathbb{C} \mid (A-\lambda) \notin \Phi_{-} (X) \rbrace ,$ 
\end{center}
\begin{center}
	$\sigma_{e k}(A) = \lbrace \lambda \in \mathbb{C} \mid (A-\lambda) \notin \Phi_{+} (X) \cup  \Phi_{-} (X) \rbrace .$ 
\end{center}
Moreover, in \cite{ZZRD}, the following has been proved:  
\begin{center}
	$\partial \sigma_{e w}(A) \subset \partial \sigma_{e a}(A) \subset \partial \sigma_{e \alpha }(A),$
\end{center}
\begin{center}
	$\partial \sigma_{e w}(A) \subset \partial \sigma_{e d}(A) \subset \partial \sigma_{e \beta }(A).$
\end{center}
Now, Hilbert $C^{*}$-modules are natural generalization of Hilbert spaces when the field of scalars is replaced by a $C^{*}$-algebra. Some important and recent results in the theory of Hilbert $C^{*}$-modules can be found in \cite{AH}, \cite{I}, \cite{L}, \cite{MT}, \cite{MSFC}, \cite{S3}, \cite{W}.\\
Fredholm theory on Hilbert $C^*$-modules as a generalization of Fredholm theory on Hilbert spaces was started by Mishchenko and Fomenko in \cite{MF}. They have worked out the notion of a Fredholm operator on the standard module $H_{\mathcal{A}} $ and proved the generalization of the Atkinson theorem.
In \cite{I} we went further in this direction and defined semi-Fredholm operators on Hilbert $C^{*}$-modules. We proved several properties of these generalized semi-Fredholm operators on Hilbert $C^{*}$-modules as an analogue or generalization of the well-known properties of classical semi-Fredholm operators on Hilbert and Banach spaces. We introduced new classes of operators on $H_{\mathcal{A}}$ 
as various generalizations of the classes $\Phi_{+}(H) ,$ $\Phi_{-}(H),$ $ \Phi_{-}^{+}(H) ,$ $ \Phi_{+}^{-}(H) $ where $H$ is a Hilbert space.\\
Inspired by these results, our idea in this paper was to obtain an analogue or generalized version of the above mentioned results in \cite{MV}, \cite{ZE}, \cite{ZZRD} in the setting of $\mathcal{A}$-Fredholm and semi-$\mathcal{A}$-Fredholm operators on Hilbert $C^{*}$-modules defined in \cite{I}, \cite{MF}.  In the third section of this paper we define other, additional classes of operators on the standard module $H_{\mathcal{A}}$ denoted by $\mathcal{M} \tilde{\Phi}_{o}(H_{\mathcal{A}}),$ $\widehat{\mathcal{M} \Phi}_{+}^{-}(H_{\mathcal{A}}),$ $\widehat{\mathcal{M} \Phi}_{-}^{+}(H_{\mathcal{A}})$ as a generalization of the classes $\Phi_{o}(H),$ $\Phi_{+}^{-}(H),$ $\Phi_{-}^{+}(H)$ and we show that these new classes of operators on $H_{\mathcal{A}}$ are open in the norm topology. \\
We consider also generalized operators of bounded, adjointable compressions on $H_{\mathcal{A}}.$ A natural generalization to Hilbert $C^{*}$-modules of compressions on Hilbert spaces is the following:\\
Let $\mathrm{P}(H_{\mathcal{A}})=\lbrace \mathrm{P} \in B(H_{\mathcal{A}}) \mid \mathrm{P}$ is the skew or the orthogonal projection and $\ker (\mathrm{P})$ is finitely generated$\rbrace .$ For $\mathrm{F} \in B(H_{\mathcal{A}})$ and $\mathrm{P} \in \mathrm{P}(H_{\mathcal{A}}),$ the compression $\mathrm{F}_{\mathrm{P}} \in B(\mathrm{Im} \mathrm{P})$ is given by $\mathrm{F}_{\mathrm{P}}=\mathrm{P}\mathrm{F}_{\mid_{\mathrm{Im} \mathrm{P}}}.$\\ 
Next, given an $\mathcal{A}$-linear, bounded, adjointable operator F on $H_{\mathcal{A}},$ we consider the operators of the form $\mathrm{F}-\alpha 1$ as $\alpha$ varies over the center of $\mathcal{A},$ denoted by  $Z(\mathcal{A}),$ and this gives rise to a different kind of spectra of F in $Z(\mathcal{A})$ as a generalization of ordinary spectra of F in $\mathbb{C} .$  \\
Using these new classes of  operators on $H_{\mathcal{A}}$ defined in \cite{MF} and \cite{I} together with the generalized spectra in $Z(\mathcal{A})$ we obtain several results as a generalization of the results from classical spectral semi-Fredholm theory given in \cite{MV}, \cite{ZE}, \cite{ZZRD}.\\
Although Fredholm and semi-Fredholm theory on Hilbert $C^{*}$-modules was established in \cite{I}, \cite{MF}, the spectral part of this theory, as an analogue or generalization of the classical spectral semi-Fredholm theory on Hilbert and Banach spaces, was not considered. The aim of this paper is to make first step into new, spectral semi-Fredholm theory on Hilbert $C^{*}$-modules in the setting of generalized spectra in the center of $C^{*}$-algebra. Our hope is that this will maybe open the way for new and further research in this direction.

\section{Preliminaries }
In this section we are going to introduce the notation, and the definitions in \cite{I} that are needed in this paper.	Throughout this paper we let $\mathcal{A}$ be a unital $\mathrm{C}^{*}$-algebra, $H_{\mathcal{A}}$ be the standard countably generated Hilbert module over $\mathcal{A}$ and we let $B^{a}(H_{\mathcal{A}})$ denote the set of all bounded , adjointable operators on $H_{\mathcal{A}}.$ Next, for $\alpha \in Z(\mathcal{A})$ we let $ \alpha \mathrm{I}$ denote the operator from $H_{\mathcal{A}}$  into $H_{\mathcal{A}}$ given by $\alpha \mathrm{I}(x)=x\alpha $ for all $x \in H_{\mathcal{A}}.$ The operator  $ \alpha \mathrm{I}$ is obviously $\mathcal{A}$-linear since $\alpha \in Z(\mathcal{A}) $ and it is adjointable with its adjoint $ \alpha^{*} \mathrm{I}.$ According to \cite[ Definition 1.4.1] {MT}, we say that a Hilbert $\mathrm{C}^*$-module $M$ over $\mathcal{A}$ is finitely generated if there exists a finite set $ \lbrace x_{i} \rbrace \subseteq M $  such that $M $ equals the linear span (over $\mathrm{C}$ and $\mathcal{A} $) of this set.\\
\begin{definition}\label{d210n} \cite[Definition 2.1]{I} 
	Let $\mathrm{F} \in B^{a}(H_{\mathcal{A}}).$ 
	We say that $\mathrm{F} $ is an upper semi-{$\mathcal{A}$}-Fredholm operator if there exists a decomposition 
	$$H_{\mathcal{A}} = M_{1} \tilde \oplus N_{1} \stackrel{\mathrm{F}}{\longrightarrow}   M_{2} \tilde \oplus N_{2}= H_{\mathcal{A}}$$
	with respect to which $\mathrm{F}$ has the matrix\\
	\begin{center}
		$\left\lbrack
		\begin{array}{cc}
		\mathrm{F}_{1} & 0 \\
		0 & \mathrm{F}_{4} \\
		\end{array}
		\right \rbrack,
		$
	\end{center}
	where $\mathrm{F}_{1}$ is an isomorphism $M_{1},M_{2},N_{1},N_{2}$ are closed submodules of $H_{\mathcal{A}} $ and $N_{1}$ is finitely generated. Similarly, we say that $\mathrm{F}$ is a lower semi-{$\mathcal{A}$}-Fredholm operator if all the above conditions hold except that in this case we assume that $N_{2}$ ( and not $N_{1}$ ) is finitely generated.	
\end{definition}
Set
\begin{center}
	$\mathcal{M}\Phi_{+}(H_{\mathcal{A}})=\lbrace \mathrm{F} \in B^{a}(H_{\mathcal{A}}) \mid \mathrm{F} $ is upper semi-{$\mathcal{A}$}-Fredholm $\rbrace ,$	
\end{center}
\begin{center}
	$\mathcal{M}\Phi_{-}(H_{\mathcal{A}})=\lbrace \mathrm{F} \in B^{a}(H_{\mathcal{A}}) \mid \mathrm{F} $ is lower semi-{$\mathcal{A}$}-Fredholm $\rbrace ,$	
\end{center}
\begin{center}
	$\mathcal{M}\Phi(H_{\mathcal{A}})=\lbrace \mathrm{F} \in B^{a}(H_{\mathcal{A}}) \mid \mathrm{F} $ is $\mathcal{A}$-Fredholm operator on $H_{\mathcal{A}}\rbrace .$
\end{center}
\begin{remark}\cite{I}
	Notice that if $M,N$ are two arbitrary Hilbert modules $\mathrm{C}^{*}$-modules, the definition above could be generalized to the classes $\mathcal{M}\Phi_{+}(M,N)$ and $\mathcal{M}\Phi_{-}(M,N)$.\\
	Recall that by \cite[ Definition 2.7.8]{MT}, originally given in \cite{MF}, when $\mathrm{F} \in \mathcal{M}\Phi(H_{\mathcal{A}})     $ and 
	$$ H_{\mathcal{A}} = M_{1} \tilde \oplus {N_{1}}\stackrel{\mathrm{F}}{\longrightarrow} M_{2} \tilde \oplus N_{2}= H_{\mathcal{A}} $$
	is an $ \mathcal{M}\Phi$-decomposition for $  \mathrm{F}   $, then the index of $\mathrm{F}$ is defined by index $ \mathrm{F}=[N_{1}]-[N_{2}] \in K(\mathcal{A})    $ where $[N_{1}]    $ and $ [N_{2}] $ denote  the isomorphism classes of $ N_{1}    $ and $ N_{2} $ respectively. By \cite[ Definition 2.7.9]{MT}, the index is well defined and does not depend on the choice of $\mathcal{M}\Phi$-decomposition for $\mathrm{F}$. As regards the $K$-group $K (\mathcal{A})$, it is worth mentioning that it is not true in general that $[M]=[N]$ implies $ M \cong N    $ for two finitely generated submodules $M, N$ of $ H_{\mathcal{A}}    $. If $K (\mathcal{A})$ satisfies the property $[N]=[M]$ implies $ N \cong M     $ for any two finitely generated, closed submodules $M, N$ of $ H_{\mathcal{A}}     $, then $K (\mathcal{A})   $ is said to satisfy "the cancellation property", see \cite[Section 6.2] {W}.	
\end{remark}
\begin{definition} \label{d280n} \cite[Definition 5.1]{I}
	Let $\mathrm{F} \in \mathcal{M}\Phi (H_{\mathcal{A}})$. 
	We say that $\mathrm{F} \in \tilde {\mathcal{M}} \Phi_{+}^{-} (H_{\mathcal{A}})$ if there exists a decomposition 
	$$H_{\mathcal{A}} = M_{1} \tilde \oplus {N_{1}}\stackrel{\mathrm{F}}{\longrightarrow} M_{2} \tilde \oplus N_{2}= H_{\mathcal{A}} $$
	with respect to which $\mathrm{F}$ has the matrix
	\begin{center}
		$\left\lbrack
		\begin{array}{cc}
		\mathrm{F}_{1} & 0 \\
		0 & \mathrm{F}_{4} \\
		\end{array}
		\right \rbrack,
		$
	\end{center}
	where $\mathrm{F}_{1}$ is an isomorphism, $N_{1},N_{2}$ are closed, finitely generated and $N_{1} \preceq N_{2},$ that is $N_{1}$ is isomorphic to a closed submodule of $N_{2}$. We define similarly the class $\tilde {\mathcal{M}}\Phi_{-}^{+} (H_{\mathcal{A}})$, the only difference in this case is that $N_{2} \preceq N_{1}$. Then we set
	$$\mathcal{M}\Phi_{+}^{-} (H_{\mathcal{A}})= (\tilde {\mathcal{M}} \Phi_{+}^{-} (H_{\mathcal{A}})) \cup (\mathcal{M}\Phi_{+} (H_{\mathcal{A}}) \setminus \mathcal{M}\Phi (H_{\mathcal{A}}))$$ 
	and
	$$\mathcal{M}\Phi_{-}^{+} (H_{\mathcal{A}})= (\tilde {\mathcal{M}}\Phi_{-}^{+} (H_{\mathcal{A}})) \cup (\mathcal{M}\Phi_{-} (H_{\mathcal{A}}) \setminus \mathcal{M}\Phi (H_{\mathcal{A}}))$$
\end{definition}
\begin{remark}
	The notation $\tilde{ \oplus} $ denotes the direct sum of modules without orthogonality, as given in \cite{MT}.\\
	At the end of this section, we also define another class of operators on $H_{\mathcal{A}}$ which we are going to use later in section 5.
\end{remark}
\begin{definition}
	We set $$ \mathcal{M}\Phi_{0}(H_{\mathcal{A}})=\lbrace \mathrm{F} \in \mathcal{M}\Phi(H_{\mathcal{A}}) \mid \textrm{ index } \mathrm{F}=0 \rbrace .$$
\end{definition}

\section{Compressions}
Throughout this paper, given an operator $\mathrm{T} \in B^{a}(H_{\mathcal{A}}),$ we let $\mathrm{R}(\mathrm{T}), \mathrm{N}(\mathrm{T}) $ denote the image of $\mathrm{T}$ and the kernel of $\mathrm{T}$ respectively. We start with the following lemma which is a generalization of \cite[Lemma 2.10.1 ]{ZZRD} originally given in \cite{PO}. The notion "projection" will in this paper generally mean skew projection, however we will specify throughout the paper whenever the projection is orthogonal.  
\begin{lemma} \label{l410} 
	Let  $\mathrm{F},\mathrm{P} \in B^{a}(H_{\mathcal{A}})  $  and suppose that  $  \mathrm{P} $  is a projection such that  $ \mathrm{N}(\mathrm{P})  $  is finitely generated. Then  $ \mathrm{F} \in \mathcal{M} \Phi (H_{\mathcal{A}})  $  if and only if  $$ \mathrm{P}\mathrm{F}_{{\mid}_{\mathrm{R}(\mathrm{P})}} \in \mathcal{M} \Phi (\mathrm{R}(\mathrm{P}))  .$$	
\end{lemma}
\begin{proof}
	Suppose first that  $\mathrm{F} \in \mathcal{M} \Phi (H_{\mathcal{A}})   $ . Observe that since  $ \mathrm{N}(\mathrm{P})  $  is finitely generated,  $ \mathrm{P} \in \mathcal{M} \Phi (H_{\mathcal{A}})   $  also.
	Hence  $ \mathrm{P}\mathrm{F}\mathrm{P} \in \mathcal{M} \Phi (H_{\mathcal{A}})   $  by \cite[Lemma 2.17 ]{MT}. \\
	Let  
	$$H_{\mathcal{A}} = M \tilde \oplus N\stackrel{\mathrm{P}\mathrm{F}\mathrm{P}}{\longrightarrow}   M^{\prime} \tilde \oplus N^{\prime}= H_{\mathcal{A}}   $$
	be a decomposition w.r.t. which  $ \mathrm{P}\mathrm{F}\mathrm{P}   $  has the matrix  
	$	\left\lbrack
	\begin{array}{cc}
	(\mathrm{P}\mathrm{F}\mathrm{P})_{1} & 0 \\
	0 & (\mathrm{P}\mathrm{F}\mathrm{P})_{4} \\
	\end{array}
	\right \rbrack
	$
	where  $ (\mathrm{P}\mathrm{F}\mathrm{P})_{1} $ is an isomorphism,  $ N, N^{\prime}   $  are finitely generated. By the proof of\\
	\cite[Theorem 2.2 ]{I} part  $ 2) \Rightarrow 1) $  we know that  $ \mathrm{P}(M)  $  is closed. Moreover, \\
	by \cite[Theorem 2.7.6 ]{MT} we may assume that $M$ is orthogonally complementable. Hence  $ \mathrm{P}_{{\mid}_{M}}  $  could be viewed as an adjointable operator from $M$ into $\mathrm{R}(\mathrm{P})$ with closed image. By \cite[Theorem 2.3.3 ]{MT} $\mathrm{P}(M)$ is then orthogonallly complementable in $\mathrm{R}(\mathrm{P}),$ that is  $ \mathrm{P}(M) \oplus \tilde{N} = \mathrm{R}(\mathrm{P})  $  for some closed submodule  $ \tilde{N}  $ . With respect to the decomposition
	$$H_{\mathcal{A}} = M \tilde \oplus N\stackrel{\mathrm{P}}{\longrightarrow}  \mathrm{P}(M) \tilde \oplus (\tilde{N} \tilde{\oplus} \mathrm{N}(\mathrm{P}))= H_{\mathcal{A}} ,  $$  
	$ \mathrm{P} $ has the matrix
	$	\left\lbrack
	\begin{array}{cc}
	\mathrm{P}_{1} &\mathrm{P}_{2} \\
	0 & \mathrm{P}_{4} \\
	\end{array}
	\right \rbrack
	,$ 
	where  $ \mathrm{P}_{1}  $  is an isomorphism. Hence  $ \mathrm{P}_{1}  $  has the matrix
	$\left\lbrack
	\begin{array}{cc}
	\mathrm{P}_{1} &0 \\
	0 & \tilde{\mathrm{P}_{4}} \\
	\end{array}
	\right \rbrack
	$   
	w.r.t. the decomposiotion
	$$H_{\mathcal{A}} = \mathrm{U}(M) \tilde \oplus \mathrm{U}(N)\stackrel{\mathrm{P}}{\longrightarrow}   \mathrm{P}(M) \tilde \oplus (\tilde{N} \tilde{\oplus} \mathrm{N}(\mathrm{P}))= H_{\mathcal{A}} ,  $$ 
	where $\mathrm{U}$ has the matrix
	$\left\lbrack
	\begin{array}{cc}
	1&-\mathrm{P}_{1}^{-1}\mathrm{P}_{2} \\
	0 & 1 \\
	\end{array}
	\right \rbrack
	$   
	w.r.t. the decomposition\\
	$M\tilde{\oplus} N\stackrel{\mathrm{U}}{\longrightarrow} M\tilde{\oplus} N,$  so that $\mathrm{U}$ is an isomorphism. Since  $\mathrm{P} \in \mathcal{M} \Phi (H_{\mathcal{A}})   $  and  $  \mathrm{U}(N)  $  is finitely generated, by \cite[Lemma 2.16 ]{I} it follows that $ \tilde{N} \tilde{\oplus} \mathrm{N}(\mathrm{P})  $  is finitely generated. Hence  $ \tilde{N}  $  is finitely generated. Now,  $ \mathrm{P}\mathrm{F}_{{\mid}_{\mathrm{P}(M)}}  $  is an isomorhism from $\mathrm{P}(M)$ onto  $ M^{\prime}  $ . Since $\mathrm{P}(M)$ is also orthogonally complementable in  $ H_{\mathcal{A}}  $  (because $\mathrm{P}_{{\mid}_{M}} \in  B^{a}(M,H_{\mathcal{A}})    ,$ as $M$ is orthogonally complementable, $ \mathrm{P}$ is adjointable   and $\mathrm{P}(M)$ is closed), it follows again that $\mathrm{P}\mathrm{F}_{{\mid}_{\mathrm{P}(M)}}   $  can be viewed as an adjointable operator from $\mathrm{P}(M)$ into $\mathrm{R}(\mathrm{P})$, so $M^{\prime}$ is orthogonally complementable in $\mathrm{R}(\mathrm{P})$ by \cite[Theorem 2.3.3 ]{MT} (since  $M^{\prime}=\mathrm{R}(\mathrm{P}\mathrm{F}_{{\mid}_{\mathrm{P}(M)}})   .$ Thus $M^{\prime}  \tilde{\oplus} \tilde{N}^{\prime}=\mathrm{R}(\mathrm{P}) $ for some closed submodule  $\tilde{N}^{\prime}  .$ Now,
	$$  H_{\mathcal{A}} = M^{\prime} \tilde \oplus N^{\prime} = M^{\prime} \tilde \oplus  \tilde{N}^{\prime}   \tilde \oplus \mathrm{N}(\mathrm{P}),$$  
	so it follows that  $(\tilde{N}^{\prime}   \tilde \oplus \mathrm{N}(\mathrm{P})) \cong N^{\prime}   .$ Since  $N^{\prime}   $  is finitely generated, it follows that  $ \tilde{N}  $  is finitely generated also. With respect to the decomposition
	$$ \mathrm{R}(\mathrm{P}) = \mathrm{P}(M) \oplus \tilde{N}\stackrel{\mathrm{P}\mathrm{F}}{\longrightarrow}  M^{\prime}  \oplus \tilde{N}^{\prime}= \mathrm{R}(\mathrm{P}) ,$$  
	${\mathrm{P}\mathrm{F}_{{\mid}_{\mathrm{R}(\mathrm{P})}}} $ has the matrix
	$\left\lbrack
	\begin{array}{cc}
	(\mathrm{P}\mathrm{F})_{1}&(\mathrm{P}\mathrm{F})_{2} \\
	0 &(\mathrm{P}\mathrm{F})_{4} \\
	\end{array}
	\right \rbrack
	,$    
	where  $(\mathrm{P}\mathrm{F})_{1}$  is an isomorphism. Then  $ {\mathrm{P}\mathrm{F}_{{\mid}_{\mathrm{R}(\mathrm{P})}}}  $  has the matrix
	$\left\lbrack
	\begin{array}{cc}
	(\mathrm{P}\mathrm{F})_{1}&0 \\
	0 & (\mathrm{P}\mathrm{F})_{4} \\
	\end{array}
	\right \rbrack
	$     
	w.r.t. the decomposition
	$$ \mathrm{R}(\mathrm{P}) = \tilde{\mathrm{U}}(\mathrm{P}(M)) \tilde{\oplus} \tilde{\mathrm{U}}(\tilde{N})\stackrel{\mathrm{P}\mathrm{F}}{\longrightarrow}  M^{\prime}  \oplus \tilde{N}^{\prime}= \mathrm{R}(\mathrm{P}) ,$$   
	where  $ \tilde{\mathrm{U}}  $ is an isomorphism of $\mathrm{R}(\mathrm{P})$ onto $\mathrm{R}(\mathrm{P}).$  Since  $ \tilde{N}, \tilde{N}^{\prime}  $  and thus also  $\tilde{\mathrm{U}}(\tilde{N})   $  are finitely generated, it follows that  $\mathrm{P}\mathrm{F}_{{\mid}_{\mathrm{R}(\mathrm{P})}} \in \mathcal{M} \Phi(\mathrm{R}(\mathrm{P}))   .$\\
	Conversely, suppose that  $ \mathrm{P}\mathrm{F}_{{\mid}_{\mathrm{R}(\mathrm{P})}} \in \mathcal{M} \Phi(\mathrm{R}(\mathrm{P}))  .$\\
	Let
	$$\mathrm{R}(\mathrm{P}) = M \tilde \oplus N\stackrel{\mathrm{P}\mathrm{F}}{\longrightarrow}   M^{\prime} \tilde \oplus N^{\prime}= \mathrm{R}(\mathrm{P})  $$
	be a decomposition w.r.t. which  $\mathrm{P}\mathrm{F}_{{\mid}_{\mathrm{R}(\mathrm{P})}}   $  has the matrix
	$\left\lbrack
	\begin{array}{cc}
	(\mathrm{P}\mathrm{F})_{1}&0 \\
	0 & (\mathrm{P}\mathrm{F})_{4} \\
	\end{array}
	\right \rbrack
	,$   
	where  $ N, N^{\prime}   $  are finitely generated and $(\mathrm{P}\mathrm{F})_{1}$, is an isomorphism. It follows that w.r.t. the decomposition
	$$  H_{\mathcal{A}} = M \tilde \oplus (N \tilde{\oplus} \mathrm{N}(\mathrm{P}))\stackrel{\mathrm{F}}{\longrightarrow} M^{\prime} \tilde \oplus (N^{\prime} \tilde{\oplus} \mathrm{N}(\mathrm{P}))= H_{\mathcal{A}} ,$$
	$ \mathrm{F} $ has the matrix
	$\left\lbrack
	\begin{array}{cc}
	\mathrm{F}_{1} & \mathrm{F}_{2} \\
	\mathrm{F}_{3} & \mathrm{F}_{4} \\
	\end{array}
	\right \rbrack
	,$  
	where  $ \mathrm{F}_{1}  $  is an isomorphism as  $ \mathrm{F}_{1}=(\mathrm{P}\mathrm{F})_{1}  $ . Indeed,  $\mathrm{F}_{1}=\sqcap_{M^{\prime}} \mathrm{F}_{{\mid}_{M}}   ,$  where  $ \sqcap_{M^{\prime}}   $  denotes the projection onto  $M^{\prime}   $  along  $ N^{\prime} \tilde{\oplus} \mathrm{N}(\mathrm{P})  .$ But then, since $\mathrm{P}\mathrm{F}$ maps $M$ isomorphically onto  $M^{\prime}   $  and so $\mathrm{R}(\mathrm{P}) = M^{\prime} \tilde{\oplus}  N^{\prime} $ , it follows that  $\mathrm{P}\mathrm{F}_{{\mid}_{M}}=\sqcap_{M^{\prime}} \mathrm{F}_{{\mid}_{M}}    $ . Hence $\mathrm{F}_{1}= \sqcap_{M^{\prime}} \mathrm{F}_{{\mid}_{M}} =\mathrm{P}\mathrm{F}_{{\mid}_{M}}  $  is an isomorphism from $M$ onto  $ M^{\prime}  .$ Using the techniques of diagonalization from the proof of \\
	\cite[Lemma 2.7.10]{MT} and the fact that  $ N \tilde{\oplus} \mathrm{N}(\mathrm{P})  $  and  $  N^{\prime} \tilde{\oplus} \mathrm{N}(\mathrm{P}) $  are finitely generated, one deduces that  $ \mathrm{F} \in \mathcal{M} \Phi (H_{\mathcal{A}})  .$ 	
\end{proof}
\begin{flushright}
	$\boxdot$
\end{flushright}
\begin{corollary} \label{c420} 
	Let  $\mathrm{F},\mathrm{P} \in B^{a}(H_{\mathcal{A}})   $  and suppose that $\mathrm{P}$ is a projection such that  $ \mathrm{N}(\mathrm{P})   $  is finitely generated. Then  
	$\sigma_{e}^{\mathcal{A}}(\mathrm{F})=\sigma_{e}^{\mathcal{A}} (\mathrm{P}\mathrm{F}_{{\mid}_{\mathrm{P}(M)}})  $ where
	$$\sigma_{e}^{\mathcal{A}} (\mathrm{P}\mathrm{F}_{{\mid}_{\mathrm{R}(\mathrm{P})}})= \lbrace \alpha \in Z(\mathcal{A}) \mid (\mathrm{P}\mathrm{F}-\alpha \mathrm{I} )_{{\mid}_{\mathrm{R}(\mathrm{P})}} \notin \mathcal{M} \Phi(\mathrm{R}(\mathrm{P})) \rbrace .$$ 
\end{corollary}
Let now $ \tilde{\mathcal{M} \Phi}_{0} (H_{\mathcal{A}})$ be the set of all $ \mathrm{F} \in B^{a}(H_{\mathcal{A}}) $ such that there exists a decomposition  
$$H_{\mathcal{A}} = M_{1} \tilde \oplus {N_{1}}\stackrel{\mathrm{F}}{\longrightarrow}  M_{2} \tilde \oplus N_{2}= H_{\mathcal{A}}   $$  
w.r.t. which $\mathrm{F}$ has the matrix
$\left\lbrack
\begin{array}{cc}
\mathrm{F}_{1} & 0 \\
0 & \mathrm{F}_{4} \\
\end{array}
\right \rbrack
$,  
where  $ \mathrm{F}_{1}  $  is an isomorphism,  $ N_{1},N_{2}  $  are finitely generated and   
$$ N \tilde \oplus {N_{1}} = N \tilde \oplus N_{2}= H_{\mathcal{A}}   $$ 
for some closed submodule  $ N \subseteq H_{\mathcal{A}}.$\\
Notice that this implies that  $\mathrm{F} \in \mathcal{M} \Phi (H_{\mathcal{A}})   $  and  $ N_{1} \cong N_{2}  $ , so that index  $\mathrm{F}= [N_{1}]-[N_{2} ]=0   .$ Hence  $\tilde{\mathcal{M}} \Phi_{0}(H_{\mathcal{A}})   \subseteq  {\mathcal{M}} \Phi_{0}(H_{\mathcal{A}})   .$ \\
Let  $\mathrm{P}(H_{\mathcal{A}})=\lbrace \mathrm{P} \in B(H_{\mathcal{A}}) \mid     $  $\mathrm{P}$ is a projection and  $ \mathrm{N}(\mathrm{P})  $  is finitely generated$\rbrace$ \\
and let
$$\sigma_{e\mathrm{W}}^{\mathcal{A}} (\mathrm{F})= \lbrace \alpha \in Z(\mathcal{A}) \mid (\mathrm{F}-\alpha \mathrm{I} ) \notin \tilde {\mathcal{M}} \Phi_{0}(H_{\mathcal{A}}) \rbrace $$  
for  $ \mathrm{F}\in B^{a}(H_{\mathcal{A}})   $ . Then we have the following theorem.
\begin{theorem} \label{t430} 
	Let  $ \mathrm{F}\in B^{a}(H_{\mathcal{A}})   $ . Then
	$$\sigma_{e\mathrm{W}}^{\mathcal{A}} (\mathrm{F})=\cap \lbrace  \sigma^{\mathcal{A}}  (\mathrm{P}\mathrm{F}_{{\mid}_{\mathrm{R}(\mathrm{P})}}) \mid  \mathrm{P}\in \mathrm{P}(H_{\mathcal{A}}) \rbrace$$  
	where\\
	$ \sigma^{\mathcal{A}}  (\mathrm{P}\mathrm{F}_{{\mid}_{\mathrm{R}(\mathrm{P})}}) =\lbrace  \alpha \in Z(\mathcal{A}) \mid (\mathrm{P}\mathrm{F}-\alpha \mathrm{I} )_{{\mid}_{\mathrm{R}(\mathrm{P})}}     $
	is not invertible in  $ B(\mathrm{R}(\mathrm{P})) \rbrace .$
\end{theorem}
\begin{proof}
	Let  $\alpha \notin \cap \lbrace \sigma^{\mathcal{A}}  (\mathrm{P}\mathrm{F}_{{\mid}_{\mathrm{R}(\mathrm{P})}}) \mid   \mathrm{P}\in \mathrm{P}(H_{\mathcal{A}}) \rbrace  $ . Then there exists some\\
	$ \mathrm{P}\in \mathrm{P}(H_{\mathcal{A}})  $  such that  $(\mathrm{P}\mathrm{F}-\alpha \mathrm{I})_{{\mid}_{\mathrm{R}(\mathrm{P})}}   $  is invertible in  $ B(\mathrm{R}(\mathrm{P}))  $ . Hence\\
	$(\mathrm{P}\mathrm{F}-\alpha \mathrm{I})_{{\mid}_{\mathrm{R}(\mathrm{P})}}    $  is an isomorphism from $\mathrm{R}(\mathrm{P})$ onto $\mathrm{R}(\mathrm{P})$, so w.r.t. the decomposition  
	$$H_{\mathcal{A}} = \mathrm{R}(\mathrm{P}) \tilde \oplus \mathrm{N}(\mathrm{P})\stackrel{\mathrm{F}-\alpha \mathrm{I}}{\longrightarrow}   \mathrm{R}(\mathrm{P}) \tilde \oplus \mathrm{N}(\mathrm{P})= H_{\mathcal{A}},$$  
	$\mathrm{F}-\alpha \mathrm{I}$ has the matrix
	$\left\lbrack
	\begin{array}{cc}
	(\mathrm{F}-\alpha \mathrm{I})_{1} & (\mathrm{F}-\alpha \mathrm{I})_{2} \\
	(\mathrm{F}-\alpha \mathrm{I})_{3} & (\mathrm{F}-\alpha \mathrm{I})_{4} \\
	\end{array}
	\right \rbrack
	,$
	where  \\
	$ (\mathrm{F}-\alpha \mathrm{I})_{1} =$ 
	$(\mathrm{P}\mathrm{F}-\alpha \mathrm{I})_{{\mid}_{\mathrm{R}(\mathrm{P})}} $  
	is an isomorphism. Then, w.r.t. the decomposition
	$$H_{\mathcal{A}} =\mathrm{U}( \mathrm{R}(\mathrm{P})) \tilde \oplus \mathrm{U}(\mathrm{N}(\mathrm{P}))\stackrel{\mathrm{F}-\alpha \mathrm{I}}{\longrightarrow}   \mathrm{V}^{-1}(\mathrm{R}(\mathrm{P})) \tilde \oplus \mathrm{V}^{-1}(\mathrm{N}(\mathrm{P}))= H_{\mathcal{A}} ,$$  
	$\mathrm{F}-\alpha \mathrm{I}$ has the matrix
	$\left\lbrack
	\begin{array}{cc}
	\overbrace{(\mathrm{F}-\alpha \mathrm{I})_{1}} & 0 \\
	0 & \overbrace{(\mathrm{F}-\alpha \mathrm{I})_{4}} \\
	\end{array}
	\right \rbrack
	,$
	where
	\begin{center}
		$\mathrm{U}$ has the matrix $\left\lbrack
		\begin{array}{cc}
		1 & -(\mathrm{F}-\alpha \mathrm{I})_{1}^{-1}(\mathrm{F}-\alpha \mathrm{I})_{2} \\
		0 & 1 \\
		\end{array}
		\right \rbrack
		,$
		w.r.t. the decomposition
		$\mathrm{R}(\mathrm{P}) \tilde \oplus \mathrm{N}(\mathrm{P})\stackrel{\mathrm{U}}{\longrightarrow}  \mathrm{R}(\mathrm{P}) \tilde \oplus \mathrm{N}(\mathrm{P}),$
	\end{center} 
	\begin{center}
		$\mathrm{V}$ has the matrix $\left\lbrack
		\begin{array}{cc}
		1 &  0\\
		-(\mathrm{F}-\alpha \mathrm{I})_{3}(\mathrm{F}-\alpha \mathrm{I})_{1}^{-1} & 1 \\
		\end{array}
		\right \rbrack
		$, w.r.t. the decomposition 
		$\mathrm{R}(\mathrm{P}) \tilde \oplus \mathrm{N}(\mathrm{P})\stackrel{\mathrm{V}}{\longrightarrow}  \mathrm{R}(\mathrm{P}) \tilde \oplus \mathrm{N}(\mathrm{P}),$
	\end{center} 
	so $\mathrm{U},\mathrm{V}$ are isomorphisms and  ${\overbrace{(\mathrm{F}-\alpha \mathrm{I})}}_{1}$  is an isomorphism.\\
	Notice that  $\mathrm{U}( \mathrm{R}(\mathrm{P})) =   \mathrm{R}(\mathrm{P}), \mathrm{V}^{-1}(\mathrm{N}(\mathrm{P}))=\mathrm{N}(\mathrm{P})  $ . Set  $M_{1}=\mathrm{R}(\mathrm{P}),$ \\
	$N_{1}=\mathrm{U}(\mathrm{N}(\mathrm{P})),M_{2}=\mathrm{V}^{-1}(\mathrm{R}(\mathrm{P})), N_{2}=\mathrm{N}(\mathrm{P})$  and  $N=\mathrm{R}(\mathrm{P}).$ It follows that  $(\mathrm{F}-\alpha \mathrm{I}) \in \tilde{\mathcal{M}} \Phi_{0}(H_{\mathcal{A}})   ,$ so  $ \alpha \notin \sigma_{e\mathrm{W}}^{\mathcal{A}} (\mathrm{F})   $. \\
	Conversely, suppose that  $ \alpha \notin \sigma_{e\mathrm{W}}^{\mathcal{A}} (\mathrm{F})   $ . Then, by definition of  $\sigma_{e\mathrm{W}}^{\mathcal{A}} (\mathrm{F})   $  and  \\
	$ \tilde{\mathcal{M}} \Phi_{0}(H_{\mathcal{A}}) ,$  there exists a decomposition
	$$H_{\mathcal{A}} = M_{1} \tilde \oplus {N_{1}}\stackrel{\mathrm{F}-\alpha \mathrm{I}}{\longrightarrow}  M_{2} \tilde \oplus N_{2}= H_{\mathcal{A}}  $$
	w.r.t. which  $\mathrm{F}-\alpha \mathrm{I}   $  has the matrix
	$\left\lbrack
	\begin{array}{cc}
	(\mathrm{F}-\alpha \mathrm{I})_{1} & 0 \\
	0 & (\mathrm{F}-\alpha \mathrm{I})_{4} \\
	\end{array}
	\right \rbrack
	,$ 
	where  $ (\mathrm{F}-\alpha \mathrm{I})_{1}  $  is an isomorphism,  $ N_{1} , N_{2} $  are finitely generated and  $N \tilde{\oplus} N_{1} =N \tilde{\oplus} N_{2} = H_{\mathcal{A}}    $  for some closed submodule $N .$\\
	Let  $\sqcap_{M_{1}}, \sqcap_{M_{2}}  $  denote the projections onto  $M_{1}   $  along  $ N_{1}  $  and onto  $M_{2}   $  along  $ N_{2}  $  respectively. Since  $\mathrm{F}-\alpha \mathrm{I} $  has the matrix
	$\left\lbrack
	\begin{array}{cc}
	(\mathrm{F}-\alpha \mathrm{I})_{1} & 0 \\
	0 & (\mathrm{F}-\alpha \mathrm{I})_{4} \\
	\end{array}
	\right \rbrack
	$   
	w.r.t. the decomposition  
	$$H_{\mathcal{A}} = M_{1} \tilde \oplus {N_{1}}\stackrel{\mathrm{F}-\alpha \mathrm{I}}{\longrightarrow}  M_{2} \tilde \oplus N_{2}= H_{\mathcal{A}}  ,$$ it follows that  
	$$ \sqcap_{M_{2}}(\mathrm{F}-\alpha \mathrm{I})_{\mid_{N}}=(\mathrm{F}-\alpha \mathrm{I}) \sqcap_{\mid_{{M_{1}}_{\mid_{N}}} }  .$$ 
	As
	$H_{\mathcal{A}} = N \tilde \oplus {N_{1}}= M_{1} \tilde \oplus N_{1},$ 
	it follows that  $\sqcap_{\mid_{{M_{1}}_{\mid_{N}}} }    $  is an isomorphism from  $ N  $  onto  $ M_{1}  $ . Using this together with the fact that  $ (\mathrm{F}-\alpha \mathrm{I})_{\mid_{M_{1}}}  $ is an isomorphism from  $M_{1}   ,$  onto  $M_{2},$  one gets that
	$$ \sqcap_{M_{2}}(\mathrm{F}-\alpha \mathrm{I})_{\mid_{N}}=(\mathrm{F}-\alpha \mathrm{I}) \sqcap_{\mid_{{M_{1}}_{\mid_{N}}} }  $$  
	is an isomorphism from  $ N  $  onto  $  M_{2} $ . Therefore w.r.t. the decomposition \\
	$H_{\mathcal{A}} = N \tilde \oplus {N_{1}}\stackrel{\mathrm{F}-\alpha \mathrm{I}}{\longrightarrow}  M_{2} \tilde \oplus N_{2}= H_{\mathcal{A}}, \mathrm{F}-\alpha \mathrm{I}   $  has the matrix
	$\left\lbrack
	\begin{array}{cc}
	(\mathrm{F}-\alpha \mathrm{I})_{1} & 0 \\
	(\mathrm{F}-\alpha \mathrm{I})_{2} & (\mathrm{F}-\alpha \mathrm{I})_{4} \\
	\end{array}
	\right \rbrack
	,$    
	where  $ (\mathrm{F}-\alpha \mathrm{I})_{1}  $  is an isomorphism  (as  $(\mathrm{F}-\alpha \mathrm{I})_{1}=\sqcap_{M_{2}}(\mathrm{F}-\alpha \mathrm{I})_{\mid_{N}}   ).$ Hence  $\mathrm{F}-\alpha \mathrm{I}   $  has the matrix
	$\left\lbrack
	\begin{array}{cc}
	\overbrace{(\mathrm{F}-\alpha \mathrm{I})_{1}} & 0 \\
	0 & \overbrace{(\mathrm{F}-\alpha \mathrm{I})_{4}} \\
	\end{array}
	\right \rbrack
	$     
	w.r.t. the decomposition\\
	$H_{\mathcal{A}} = N \tilde \oplus {N_{1}}\stackrel{\mathrm{F}-\alpha \mathrm{I}}{\longrightarrow}  \mathrm{V}^{-1}(M_{2}) \tilde \oplus N_{2}= H_{\mathcal{A}},$ where 	$\mathrm{V}$ has the matrix
	\begin{center}
		$\left\lbrack
		\begin{array}{cc}
		1 & 0 \\
		-(\mathrm{F}-\alpha \mathrm{I})_{3}(\mathrm{F}-\alpha \mathrm{I})_{1}^{-1} & 1 \\
		\end{array}
		\right \rbrack
		,$ w.r.t. the decomposition 
		
	\end{center} 
	$  M_{2} \tilde \oplus N_{2}\stackrel{\mathrm{V}}{\longrightarrow} M_{2} \tilde \oplus N_{2},$    
	so that $\mathrm{V}$ and  ${\overbrace{(\mathrm{F}-\alpha \mathrm{I})}}_{1}   $  are isomorphisms. It follows that  $ (\mathrm{F}-\alpha \mathrm{I})_{\mid_{N}}  $ is an isomorphism from  $ N $  onto  $ \mathrm{V}^{-1}(M_{2})  .$ Next, since  $$H_{\mathcal{A}}=N \tilde \oplus N_{2}=\mathrm{V}^{-1}(M_{2}) \tilde \oplus N_{2}  ,$$ it follows that  $\mathrm{P}_{\mid_{{\mathrm{V}^{-1}(M_{2})}}}   $  is an isomorphism from  $ \mathrm{V}^{-1}(M_{2})  $  onto  $  N ,$  where  $ \mathrm{P}  $  denotes the projection onto  $ N  $  along  $ N_{2}  $ . Hence  $ \mathrm{P}(\mathrm{F}-\alpha \mathrm{I})_{\mid_{N}}  $  is an isomorphism from  $ N  $  onto  $ N, $  so 
	\begin{center} 
		$ \alpha \notin \cap \lbrace  \sigma^{\mathcal{A}}  (\mathrm{P}\mathrm{F}_{{\mid}_{\mathrm{R}(\mathrm{P})}}) \mid  \mathrm{P} \in \mathrm{P}(H_{\mathcal{A}}) $  and $\mathrm{N}(\mathrm{P})$ is finitely generated$ \rbrace.$
	\end{center}		
\end{proof}
\begin{flushright}
	$\boxdot$
\end{flushright}
\begin{lemma} \label{l440} 
	$ \tilde{\mathcal{M} \Phi}_{0}(H_{\mathcal{A}})   $  is open in $B^{a}(H_{\mathcal{A}}).$
\end{lemma}
\begin{proof}
	If $ \mathrm{F} \in \tilde{\mathcal{M} \Phi}_{0}(H_{\mathcal{A}})  $ , then there exists a decomposition
	$$H_{\mathcal{A}} = M_{1} \tilde \oplus {N_{1}}\stackrel{\mathrm{F}}{\longrightarrow} M_{2} \tilde \oplus N_{2}= H_{\mathcal{A}} $$
	w.r.t. which $\mathrm{F}$ has the matrix
	$\left\lbrack
	\begin{array}{cc}
	\mathrm{F}_{1} & 0 \\
	0 & \mathrm{F}_{4} \\
	\end{array}
	\right \rbrack
	,$  
	where  $ \mathrm{F}_{1}  $ is an isomorphism,  $ N_{1},N_{2}  $  are finitely generated and  $H_{\mathcal{A}} =N \tilde \oplus {N_{1}} = N \tilde \oplus N_{2}   $  for some closed submodule  $ N  .$ We may w.l.g. assume that  $M_{1}=N   $ . Indeed, as we have seen in the proof of the Theorem \ref{t430}, we have that  $\mathrm{P}\mathrm{F}_{{\mid}_{N}}   $  is invertible in  $ B(N)  ,$  where  $  \mathrm{P} $  is the projection onto  $ N  $  along  $ N_{2}  $ . Then, w.r.t. the decomposition
	$$H_{\mathcal{A}} = N \tilde \oplus N_{1}\stackrel{\mathrm{F}}{\longrightarrow} N \tilde \oplus N_{2}= H_{\mathcal{A}},$$  
	$ \mathrm{F} $ has the matrix
	$\left\lbrack
	\begin{array}{cc}
	\tilde{\mathrm{F}}_{1} & 0 \\
	\tilde{\mathrm{F}}_{2} &  \mathrm{F}_{4} \\
	\end{array}
	\right \rbrack
	,$   
	where  $ \tilde{\mathrm{F}}_{1}  $  is an isomorphism, so $\mathrm{F}$ has the matrix
	$\left\lbrack
	\begin{array}{cc}
	\tilde {\tilde{\mathrm{F}}}_{1} & 0 \\
	0 & \tilde{\mathrm{F}}_{4} \\
	\end{array}
	\right \rbrack
	$     
	w.r.t. the decomposition
	$$H_{\mathcal{A}} = N \tilde \oplus \tilde{\mathrm{U}}(N_{1})\stackrel{\mathrm{F}}{\longrightarrow} \tilde{\mathrm{V}}^{-1}(N) \tilde \oplus N_{2}= H_{\mathcal{A}},$$  
	where  $ \tilde {\tilde{\mathrm{F}}}_{1} ,   \tilde{\mathrm{U}}, \tilde{\mathrm{V}} $  are isomorphisms. Hence
	$$H_{\mathcal{A}} = N \tilde \oplus \tilde{\mathrm{U}}(N_{1})= N \tilde \oplus N_{2} ,$$  
	so we may assume w.l.g. that  $N=M_{1}   .$\\
	Now, by the proof of \cite[Lemma 2.7.10 ]{MT}, there exists some $\epsilon >0$ such that if $ \mathrm{D} \in B^{a}(H_{\mathcal{A}}) $ and $||\mathrm{D}-\mathrm{F}||<\epsilon ,$ then $\mathrm{D}$ has the matrix 
	$\left\lbrack
	\begin{array}{cc}
	\mathrm{D}_{1} & 0 \\
	0 & \mathrm{D}_{4} \\
	\end{array}
	\right \rbrack
	,$
	w.r.t. the decomposition 
	$$H_{\mathcal{A}} =N \tilde \oplus \mathrm{U}\tilde{\mathrm{U}}(N_{1})\stackrel{\mathrm{D}}{\longrightarrow}  \mathrm{V}^{-1}\tilde{\mathrm{V}}^{-1}(N) \tilde \oplus N_{2}= H_{\mathcal{A}} ,  $$ 
	where  $\mathrm{U},\mathrm{V} $ are isomorphisms, and $\mathrm{D}_{1}$ is an isomorphism. Since
	$$H_{\mathcal{A}} =N \tilde \oplus \mathrm{U}\tilde{\mathrm{U}}(N_{1}) =  N \tilde \oplus N_{2}= H_{\mathcal{A}} ,  $$
	it follows that $\mathrm{D} \in \tilde{\mathcal{M}} \Phi_{0}(H_{\mathcal{A}}).$
\end{proof}
\begin{flushright}
	$\boxdot$
\end{flushright}
We let now   $ \widehat{\mathcal{M} \Phi}_{+}^{-} (H_{\mathcal{A}})$  be the space of all $\mathrm{F} \in B^{a}(H_{\mathcal{A}}) $ such that there exists a decomposition
$$H_{\mathcal{A}} = M_{1} \tilde \oplus {N_{1}}\stackrel{\mathrm{F}}{\longrightarrow}  M_{2} \tilde \oplus N_{2}= H_{\mathcal{A}},$$  
w.r.t. which $\mathrm{F}$ has the matrix 
$\left\lbrack
\begin{array}{cc}
\mathrm{F}_{1} & 0 \\
0 & \mathrm{F}_{4} \\
\end{array}
\right \rbrack
,$ 
where $\mathrm{F}_{1}$ is an isomorphism,  $ N_{1}  $  is finitely generated and such that there exist closed submodules  $ N_{2}^{\prime},N  $  where  $N_{2}^{\prime} \subseteq N_{2},N_{2}^{\prime} \cong N_{1} ,$ 
$H_{\mathcal{A}}= N \tilde{\oplus} N_{1} = N \tilde{\oplus} N_{2}^{\prime} \textrm { and  the projection onto } N \textrm{ along } N_{2}^{\prime}$ is adjointable.\\
Then we set
$$\sigma_{e\tilde{a}}^{\mathcal{A}} (\mathrm{F}):= \lbrace \alpha \in Z(\mathcal{A}) \mid (\mathrm{F}-\alpha \mathrm{I} ) \notin \widehat{\mathcal{M} \Phi}_{+}^{-}(H_{\mathcal{A}}) \rbrace     .$$
\begin{theorem} \label{t450} 
	Let  $\mathrm{F} \in B^{a}(H_{\mathcal{A}})   .$ Then  $\sigma_{e\tilde{a}}^{\mathcal{A}} (\mathrm{F})=\cap \lbrace  \sigma_{a}^{\mathcal{A}}  (\mathrm{P}\mathrm{F}_{{\mid}_{\mathrm{R}(\mathrm{P})}}) \mid  \mathrm{P}\in \mathrm{P}^{a}(H_{\mathcal{A}}) \rbrace   $  
	where
	$ \sigma_{a}^{\mathcal{A}}  (\mathrm{P}\mathrm{F}_{{\mid}_{\mathrm{R}(\mathrm{P})}})$ is the set of all  $\alpha \in Z(\mathcal{A})$ s.t. $(\mathrm{P}\mathrm{F}-\alpha \mathrm{I} )_{{\mid}_{\mathrm{R}(\mathrm{P})}}$ is not bounded below on  $ \mathrm{R}(\mathrm{P})$  and $\mathrm{P}^{a}(H_{\mathcal{A}}) = \mathrm{P}(H_{\mathcal{A}}) \cap B^{a}(H_{\mathcal{A}}).$
\end{theorem}
\begin{proof}
	Suppose that  $ \alpha \notin  \sigma_{a}^{\mathcal{A}}  (\mathrm{P})   $  for some  $ \mathrm{P}\in \mathrm{P}^{a}(H_{\mathcal{A}}),  \alpha \in Z(\mathcal{A})  .$ Then the operator  $ (\mathrm{P}\mathrm{F}-\alpha \mathrm{I} )_{{\mid}_{\mathrm{R}(\mathrm{P})}}    $  is bounded below on $\mathrm{R}(\mathrm{P}),$ hence its image is closed. But  $\mathrm{R} ((\mathrm{P}\mathrm{F}-\alpha \mathrm{I} )_{{\mid}_{\mathrm{R}(\mathrm{P})}}) =\mathrm{R} (\mathrm{P}\mathrm{F}\mathrm{P}-\alpha \mathrm{P} )    .$ Since  $(\mathrm{P}\mathrm{F}\mathrm{P}-\alpha \mathrm{P} )   $  can be viewed as an adjointable operator from  $H_{\mathcal{A}}   $  into  $\mathrm{R}(\mathrm{P})   $ , from \cite[Theorem 2.3.3]{MT} it follows that   $\mathrm{R}(\mathrm{P}\mathrm{F}-\alpha \mathrm{I} )_{{\mid}_{\mathrm{R}(\mathrm{P})}} =\mathrm{R} (\mathrm{P}\mathrm{F}\mathrm{P}-\alpha \mathrm{P} )    $  is orthogonally complementable in $\mathrm{R}(\mathrm{P})$. So  $ \mathrm{R}(\mathrm{P})=M \oplus M^{\prime}  ,$  where  $ M=\mathrm{R} (\mathrm{P}\mathrm{F}-\alpha \mathrm{P} )   .$ Hence  $H_{\mathcal{A}}= M \tilde \oplus M^{\prime} \tilde \oplus \mathrm{N}(\mathrm{P}) $  and  $(\mathrm{P}\mathrm{F}-\alpha \mathrm{I} )_{{\mid}_{\mathrm{R}(\mathrm{P})}}   $  is an isomorphism from $\mathrm{R}(\mathrm{P})$ onto $ M.$ It follows that w.r.t. the decomposition
	$$H_{\mathcal{A}} = \mathrm{R}(\mathrm{P}) \tilde \oplus \mathrm{N}(\mathrm{P})\stackrel{\mathrm{F}-\alpha \mathrm{I}}{\longrightarrow}   M \tilde \oplus (\tilde{M}^{\prime} \tilde{\oplus} \mathrm{N}(\mathrm{P}))= H_{\mathcal{A}} ,   $$  
	$\mathrm{F}-\alpha \mathrm{I} $ has the matrix
	$\left\lbrack
	\begin{array}{cc}
	(\mathrm{F}-\alpha \mathrm{I})_{1} & (\mathrm{F}-\alpha \mathrm{I})_{2} \\
	(\mathrm{F}-\alpha \mathrm{I})_{3} & (\mathrm{F}-\alpha \mathrm{I})_{4} \\
	\end{array}
	\right \rbrack
	,$ 
	where  $(\mathrm{F}-\alpha \mathrm{I})_{1}   $  is an isomorphism. Hence w.r.t. the decomposition
	$$ H_{\mathcal{A}} =\mathrm{ R}(\mathrm{P}) \tilde \oplus \mathrm{U}(\mathrm{N}(\mathrm{P}))\stackrel{\mathrm{F}-\alpha \mathrm{I}}{\longrightarrow}   \mathrm{V}^{-1}(M) \tilde \oplus ({M}^{\prime} \tilde \oplus \mathrm{N}(\mathrm{P}))= H_{\mathcal{A}},   $$
	$\mathrm{F}-\alpha \mathrm{I} $ has the matrix
	$\left\lbrack
	\begin{array}{cc}
	\overbrace{(\mathrm{F}-\alpha \mathrm{I})_{1}} & 0 \\
	0 & \overbrace{(\mathrm{F}-\alpha \mathrm{I})_{4}} \\
	\end{array}
	\right \rbrack
	,$ 
	where  $\overbrace{(\mathrm{F}-\alpha \mathrm{I})_{1}},\mathrm{U},\mathrm{V}   $  are isomorphisms. Set  $N=M_{1}=\mathrm{ R}(\mathrm{P}),N_{1}=\mathrm{U}(\mathrm{N}(\mathrm{P})),M_{2}=\mathrm{V}^{-1}(M)   ,$ \\
	$N_{2}=M^{\prime}  \tilde \oplus \mathrm{N}(\mathrm{P})    $  and  $N_{2}^{\prime}=\mathrm{N}(\mathrm{P})    .$ It follows that
	$$ H_{\mathcal{A}} = N \tilde \oplus N_{1} = N \tilde \oplus N_{2}^{\prime}, N_{1} \cong N_{2}^{\prime} \subseteq N_{2}  $$  and  $\mathrm{F}-\alpha \mathrm{I}   $  has the matrix
	$\left\lbrack
	\begin{array}{cc}
	\overbrace{(\mathrm{F}-\alpha \mathrm{I})_{1}} & 0 \\
	0 & \overbrace{(\mathrm{F}-\alpha \mathrm{I})_{4}} \\
	\end{array}
	\right \rbrack
	$   
	w.r.t. the decomposition
	$ H_{\mathcal{A}} = M_{1} \tilde \oplus {N_{1}}\stackrel{\mathrm{F}-\alpha \mathrm{I}}{\longrightarrow}   M_{2} \tilde \oplus N_{2}= H_{\mathcal{A}}   $  
	where $ \overbrace{(\mathrm{F}-\alpha \mathrm{I})_{1}} $ is an isomorphism and  $N_{1}=\mathrm{N}(\mathrm{P})   $  is finitely generated. Thus  $\alpha \notin \sigma_{e\tilde{a}}^{\mathcal{A}} (\mathrm{F})   .$\\
	Conversely, suppose that
	$  \alpha \in Z(\mathcal{A}) \setminus \sigma_{e\tilde{a}}^{\mathcal{A}} (\mathrm{F}) .$ Then, there exists a decomposition
	$$ H_{\mathcal{A}} = M_{1} \tilde \oplus {N_{1}}\stackrel{\mathrm{F}-\alpha \mathrm{I}}{\longrightarrow}   M_{2} \tilde \oplus N_{2} = H_{\mathcal{A}}   $$ 
	w.r.t. which  $\mathrm{F}-\alpha \mathrm{I}   $  has the matrix
	$\left\lbrack
	\begin{array}{cc}
	(\mathrm{F}-\alpha \mathrm{I})_{1} & 0 \\
	0 & (\mathrm{F}-\alpha \mathrm{I})_{4} \\
	\end{array}
	\right \rbrack
	,$  
	where  $(\mathrm{F}-\alpha \mathrm{I})_{1}   $  is an isomorphism,  $N_{1}   $  is finitely generated and there exists some closed submodules  $ N,N_{2}^{\prime}  $  such that  $N_{2}^{\prime} \subseteq N_{2} , N_{2}^{\prime} \cong N_{1}   ,$  $N \tilde{\oplus} N_{1}=N \tilde{\oplus} N_{2}^{\prime}=H_{\mathcal{A}} $ and the projection onto $N$ along $N_{2}^{\prime} $ is adjointable. As we have seen in the proof of Theorem \ref{t430} ,  $ \sqcap_{M_{2}}(\mathrm{F}-\alpha \mathrm{I})_{\mid_{N}}  $  is an isomorphism then, where $\sqcap_{M_{2}}$ denotes the projection onto $M_{2}$ along $N_{2}$ . Therefore, w.r.t. the decomposition
	$$ H_{\mathcal{A}} = N \tilde \oplus U(N_{1})\stackrel{\mathrm{F}-\alpha \mathrm{I}}{\longrightarrow}   \mathrm{V}^{-1}(M_{2}) \tilde \oplus N_{2} = H_{\mathcal{A}}   ,$$ 
	${(\mathrm{F}-\alpha \mathrm{I})} $ has the matrix
	$\left\lbrack
	\begin{array}{cc}
	\overbrace{(\mathrm{F}-\alpha \mathrm{I})_{1}} & 0 \\
	0 & \overbrace{(\mathrm{F}-\alpha \mathrm{I})_{4}} \\
	\end{array}
	\right \rbrack
	,$  
	where  $\overbrace{(\mathrm{F}-\alpha \mathrm{I})_{1}} ,U,\mathrm{V}   $  are isomorphisms. Hence  $(\mathrm{F}-\alpha \mathrm{I})_{\mid_{N}}   $  maps $N$ isomorphically onto  $\mathrm{V}^{-1}(M)   .$ Since 
	$N_{2}^{\prime} \cong N_{1} ,  $  it follows that  $ N_{2}^{\prime}   $  is finitely generated (as $N_{1}   $  is so), hence, \\
	by \cite[Lemma 2.3.7]{MT}, as $ N_{2}^{\prime}$   is a closed submodule  of  $N_{2} ,$ we get that \\
	$N_{2}=N_{2}^{\prime} \tilde \oplus {N_{2}^{\prime}}^{\prime} $ for some closed submodule ${N_{2}^{\prime}}^{\prime}  $ of $N_{2}.$ So  
	$$ H_{\mathcal{A}} =\mathrm{V}^{-1} (M_{2}) \oplus N_{2} = \mathrm{V}^{-1} (M_{2}) \tilde \oplus {N_{2}^{\prime}}^{\prime} \tilde \oplus N_{2}^{\prime}= N \tilde \oplus  N_{2}^{\prime}.$$ 
	It follows that if  $ \mathrm{P}  $  is the projection onto $N$ along  $ N_{2}^{\prime}  ,$ then  $ \mathrm{P}_{\mid_{\mathrm{V}^{-1} (M_{2}) \tilde \oplus {N_{2}^{\prime}}^{\prime}}}  $  
	is an isomorphism from  $\mathrm{V}^{-1} (M_{2}) \tilde \oplus {N_{2}^{\prime}}^{\prime}   $  onto $N$. Hence  $\mathrm{P}_{\mid_{\mathrm{V}^{-1} (M_{2})}}  $  maps  $ \mathrm{V}^{-1} (M_{2})  $  isomorphically onto some closed submodule of $N.$ Using this together with the fact that  $(\mathrm{F}-\alpha \mathrm{I})_{\mid_{N}}   ,$ is an isomorphism from $ N $ onto  $\mathrm{V}^{-1} (M_{2})  , $  we obtain that  $\mathrm{P}(\mathrm{F}-\alpha \mathrm{I})_{\mid_{N}}    $  is bounded below. Thus  $\alpha \notin \sigma_{{a}}^{\mathcal{A}} (\mathrm{P}\mathrm{F}_{\mid_{\mathrm{ R}(\mathrm{P})}}).$
\end{proof}
\begin{flushright}
	$\boxdot$
\end{flushright}
\begin{remark}
	In the similar way as for $\tilde{\mathcal{M} \Phi_{0}}(H_{\mathcal{A}})$, one can show that\\
	$\widehat{\mathcal{M} \Phi}_{+}^{-}(H_{\mathcal{A}})$ is open in $B^{a}(H_{\mathcal{A}}) .$	
\end{remark}
Indeed, let $F \in \widehat{\mathcal{M} \Phi}_{+}^{-}(H_{\mathcal{A}})  $  and choose a decomposition $$H_{\mathcal{A}} = M_{1} \tilde \oplus {N_{1}}\stackrel{F}{\longrightarrow} M_{2} \tilde \oplus N_{2}= H_{\mathcal{A}}  $$ w.r.t. which $F$ has the matrix 
$\left\lbrack
\begin{array}{cc}
F_{1} & 0 \\
0 & F_{4} \\
\end{array}
\right \rbrack,
$ where $F_{1}  $ is an isomorphism. Let $N$ be a closed submodule of $H_{\mathcal{A}}  $ s.t. $H_{\mathcal{A}} = N \tilde \oplus {N_{1}}= N \tilde \oplus N_{2}^{\prime},$ $N_{2}^{\prime} \subseteq  N_{2},$ and the projection onto $N$ along $N_{2}^{\prime}  $ is adjointable. Such decomposition exists as $ F \in \widehat{\mathcal{M} \Phi}_{+}^{-}(H_{\mathcal{A}})  .$ It is easy to see that, if we let $\sqcap_{M_{1}},\sqcap_{M_{2}}  $ denote the projections onto $M_{1}  $ along $N_{1}  $ and onto $ M_{2} $ along $N_{2}  ,$ respectively, then $\sqcap_{M_{2}} F_{\mid_{N}}=F\sqcap_{{M_{1}}_{\mid_{N}}}   $ is an isomorphism. Hence, w.r.t. the decomposition $H_{\mathcal{A}} = N \tilde \oplus {N_{1}}\stackrel{F}{\longrightarrow} M_{2} \tilde \oplus N_{2}= H_{\mathcal{A}}  ,$ $F$ has the matrix 
$\left\lbrack
\begin{array}{cc}
\tilde{{F}_{1}} & 0 \\
\tilde{{F}_{2}} & F_{4} \\
\end{array}
\right \rbrack,
$ 
where $\tilde{{F}_{1}}  $ is an isomorphism. Then, using the techniques of diagonalization as in the proof of \cite[Lemma 2.7.10]{MT}, we get that $F$ has the matrix 
$\left\lbrack
\begin{array}{cc}
\tilde{\tilde{{F}_{1}}} & 0 \\
0 & F_{4} \\
\end{array}
\right \rbrack,
$ 
w.r.t. the decomposition $H_{\mathcal{A}} = N \tilde \oplus {N_{1}}\stackrel{F}{\longrightarrow} V^{-1}(M_{2}) \tilde \oplus N_{2}= H_{\mathcal{A}}   ,$ where $V,\tilde{\tilde{{F}_{1}}}   $ are isomorphisms. By the proof of \cite[Lemma 2.7.10]{MT}, there exists an $\epsilon >0  $ s.t. if $ \parallel F-D \parallel < \epsilon ,$ then $D$ has the matrix 
$\left\lbrack
\begin{array}{cc}
D_{1} & 0 \\
0 & D_{4} \\
\end{array}
\right \rbrack,
$  
w.r.t. the decomposition $H_{\mathcal{A}} = N \tilde \oplus U({N_{1}})\stackrel{D}{\longrightarrow} \tilde{V}^{-1}V^{-1}(M_{2}) \tilde \oplus N_{2}= H_{\mathcal{A}} ,$ where $U,\tilde{V},D_{1}  $ are isomorphisms. Since $H_{\mathcal{A}} = N \tilde \oplus U({N_{1}})= N \tilde \oplus N_{2}^{\prime},$ $N_{2}^{\prime} \subseteq N_{2} $ and the projection onto $N$ along $N_{2}^{\prime}  $ is adjointable,  it follows that $D \in \widehat{\mathcal{M} \Phi}_{+}^{-}(H_{\mathcal{A}}).$
\begin{definition} \label{d470} 
	We set $\widehat{\mathcal{M} \Phi}_{-}^{+}(H_{\mathcal{A}})$ to be the set of all $ \mathrm{D} \in B^{a}  (H_{\mathcal{A}}) $ such that there exists a decomposition
	$$H_{\mathcal{A}} = M_{1}^{\prime} \tilde \oplus {N_{1}^{\prime}}\stackrel{\mathrm{D}}{\longrightarrow} M_{2}^{\prime} \tilde \oplus N_{2}^{\prime}= H_{\mathcal{A}} $$ 
	w.r.t. which $\mathrm{D}$ has the matrix 
	$\left\lbrack
	\begin{array}{cc}
	\mathrm{D}_{1} & 0 \\
	0 & \mathrm{D}_{4}  \\
	\end{array}
	\right \rbrack
	,$ where $\mathrm{D}_{1}$ is an isomorphism, $N_{2}^{\prime}$ is finitely generated and such that $H_{\mathcal{A}}=M_{1}^{\prime} \tilde \oplus N \tilde \oplus N_{2}^{\prime}$ for some closed submodule $N,$ where the projection onto $M_{1}^{\prime} \tilde \oplus N $ along $N_{2}^{\prime}$ is adjointable.
\end{definition}
Then we set 
$$\sigma_{e \tilde{d}}^{\mathcal{A}}(\mathrm{D})=\lbrace \alpha \in Z(\mathcal{A})) \mid (\mathrm{D}-\alpha \mathrm{I}) \notin \widehat{\mathcal{M} \Phi}_{-}^{+}(H_{\mathcal{A}}) \rbrace$$
and for $\mathrm{P} \in \mathrm{P}^{a}(H_{\mathcal{A}})$ we set 
$$\sigma_{d}^{\mathcal{A}}(\mathrm{P}\mathrm{D}_{\mid_{\mathrm{ R}(\mathrm{P})}})=\lbrace \alpha \in Z(\mathcal{A})) \mid (\mathrm{P}\mathrm{D}-\alpha \mathrm{I})_{\mid_{\mathrm{ R}(\mathrm{P})}} \textrm{ is not onto }\mathrm{ R}(\mathrm{P}) \rbrace .$$
We have then the following theorem.
\begin{theorem} \label{t480} 
	Let $\mathrm{D} \in B^{a}(H_{\mathcal{A}}).$ Then 
	$$\sigma_{e \tilde{d}}^{\mathcal{A}}(\mathrm{D})= \bigcap \lbrace \sigma_{d}^{\mathcal{A}} (\mathrm{P}\mathrm{D}_{\mid_{\mathrm{ R}(\mathrm{P})}}) \mid \mathrm{P} \in \mathrm{P}^{a} (H_{\mathcal{A}}) \rbrace   $$
\end{theorem}
\begin{proof}
	Suppose first that $\alpha \notin \bigcap \lbrace \sigma_{d}^{\mathcal{A}} (\mathrm{P}\mathrm{D}_{\mid_{\mathrm{ R}(\mathrm{P})}}) \mid \mathrm{P} \in \mathrm{P}^{a} (H_{\mathcal{A}}) \rbrace      ,$ then\\
	$(\mathrm{P}\mathrm{D}-\alpha \mathrm{I})_{\mid_{\mathrm{ R}(\mathrm{P})}}$ is onto $\mathrm{ R}(\mathrm{P})$ for some $\mathrm{P} \in \mathrm{P}^{a} (H_{\mathcal{A}}).$ Since $\mathrm{P}$ is adjointable and $\mathrm{ R}(\mathrm{P})$ is closed, by \cite[Theorem 2.3.3]{MT} $\mathrm{ R}(\mathrm{P})$ is orthogonally complementable in $H_{\mathcal{A}},$ hence $(\mathrm{P}\mathrm{D}-\alpha \mathrm{I})_{\mid_{\mathrm{ R}(\mathrm{P})}}$ can be viewed as an adjointable operator from $\mathrm{ R}(\mathrm{P})$ onto $\mathrm{ R}(\mathrm{P}).$ Then, again by \cite[Theorem 2.3.3]{MT}, $ \mathrm{N} (\mathrm{P}\mathrm{D}-\alpha \mathrm{I})_{\mid_{\mathrm{ R}(\mathrm{P})}}$ is orthogonally complementable in $\mathrm{ R}(\mathrm{P}),$ that is $\mathrm{ R}(\mathrm{P})= (\mathrm{N} (\mathrm{P}\mathrm{D}-\alpha \mathrm{I})_{\mid_{\mathrm{ R}(\mathrm{P})}}) \oplus \tilde{N}$ for some closed submodule $\tilde{N}.$ The operator $\mathrm{P}\mathrm{D}-\alpha \mathrm{I}$ is an isomorphism from $\tilde{N}$ onto $\mathrm{ R}(\mathrm{P}).$ Hence w.r.t. the decomposition 
	$$H_{\mathcal{A}}=\tilde{N} \tilde{\oplus} ( (\mathrm{N} (\mathrm{P}\mathrm{D}-\alpha \mathrm{I})_{\mid_{\mathrm{ R}(\mathrm{P})}}) \tilde{\oplus} \mathrm{N}(\mathrm{P}))\stackrel{\mathrm{D}-\alpha \mathrm{I}}{\longrightarrow}\mathrm{ R}(\mathrm{P})  \tilde{\oplus} \mathrm{N}(\mathrm{P}) =  H_{\mathcal{A}},$$ $\mathrm{D}-\alpha \mathrm{I}$ has the matrix
	$\left\lbrack
	\begin{array}{cc}
	(\mathrm{D}-\alpha \mathrm{I})_{1} & (\mathrm{D}-\alpha \mathrm{I})_{2} \\
	(\mathrm{D}-\alpha \mathrm{I})_{3} & (\mathrm{D}-\alpha \mathrm{I})_{4} \\
	\end{array}
	\right \rbrack,
	$
	where $(\mathrm{D}-\alpha \mathrm{I})_{1}$ is an isomorphism. It follows that $\mathrm{D}-\alpha \mathrm{I}$ has the matrix
	$\left\lbrack
	\begin{array}{cc}
	\overbrace{(\mathrm{D}-\alpha \mathrm{I})_{1}} & 0 \\
	0 & \overbrace{(\mathrm{D}-\alpha \mathrm{I})_{4}} \\
	\end{array}
	\right \rbrack,
	$
	w.r.t. the decomposition 
	$$H_{\mathcal{A}}=\tilde{N} \tilde{\oplus} U( (\mathrm{N} (\mathrm{P}\mathrm{D}-\alpha \mathrm{I})_{\mid_{\mathrm{ R}(\mathrm{P})}})  \tilde{\oplus} \mathrm{N}(\mathrm{P}))\stackrel{\mathrm{D}-\alpha \mathrm{I}}{\longrightarrow} \mathrm{V}^{-1}(\mathrm{ R}(\mathrm{P})) \tilde{\oplus} \mathrm{N}(\mathrm{P}) = H_{\mathcal{A}},$$
	where $\mathrm{U},\mathrm{V}$ and $\overbrace{(\mathrm{D}-\alpha \mathrm{I})_{1}} $   are isomorphisms. Set $N= \mathrm{N} (\mathrm{P}\mathrm{D}-\alpha \mathrm{I})_{\mid_{\mathrm{ R}(\mathrm{P})}} ,$\\
	$ M_{1}=\tilde{N}, M_{2}^{\prime}= \mathrm{V}^{-1}(\mathrm{ R}(\mathrm{P})), N_{1}^{\prime}=\mathrm{U}( (\mathrm{N} (\mathrm{P}\mathrm{D}-\alpha \mathrm{I})_{\mid_{\mathrm{ R}(\mathrm{P})}})  \tilde{\oplus} \mathrm{N}(\mathrm{P})),N_{2}=\mathrm{N}(\mathrm{P})$ and observe that $\mathrm{ R}(\mathrm{P})=N \oplus \tilde{N}.$ Hence $(\mathrm{D}-\alpha \mathrm{I}) \in \widehat{\mathcal{M} \Phi}_{-}^{+}(H_{\mathcal{A}}).$\\
	Conversely, let $\alpha \notin \sigma_{e \tilde{d}}^{\mathcal{A}}(\mathrm{D})$ and let 
	$$H_{\mathcal{A}} = M_{1}^{\prime} \tilde \oplus {N_{1}^{\prime}}\stackrel{\mathrm{D}-\alpha \mathrm{I}}{\longrightarrow} M_{2}^{\prime} \tilde \oplus N_{2}^{\prime}= H_{\mathcal{A}} $$ 
	be decomposition w.r.t. which $\mathrm{D}-\alpha \mathrm{I}$ has the matrix
	$\left\lbrack
	\begin{array}{cc}
	(\mathrm{D}-\alpha \mathrm{I})_{1} & 0 \\
	0 & (\mathrm{D}-\alpha \mathrm{I})_{4} \\
	\end{array}
	\right \rbrack,
	$ 
	where $(\mathrm{D}-\alpha \mathrm{I})_{1}$ is an isomorphism and such that $H_{\mathcal{A}}=M_{1}^{\prime} \tilde{\oplus} N \tilde{\oplus} N_{2}^{\prime}$ for some closed submodule $N,$ where the projection onto $M_{1}^{\prime} \tilde{\oplus} N $ along $N_{2}^{\prime}$ is adjointable. It follows that $\mathrm{P}_{\mid_{M_{2}^{\prime}}}$ is an isomorphism onto $M_{1}^{\prime} \tilde{\oplus} N,$ where $\mathrm{P}$ is the projection onto $M_{1}^{\prime} \tilde{\oplus} N $ along $N_{2}^{\prime}.$ Hence $\mathrm{P}(\mathrm{D}-\alpha \mathrm{I})_{\mid_{M_{1}^{\prime} }}$ is an isomorphism onto $M_{1}^{\prime} \tilde{\oplus} N.$ Therefore $\mathrm{P}(\mathrm{D}-\alpha \mathrm{I})_{\mid_{M_{1}^{\prime} \tilde{\oplus} N}}$ is onto $M_{1}^{\prime} \tilde{\oplus} N.$ Now $\mathrm{ R}(\mathrm{P})=M_{1}^{\prime} \tilde{\oplus} N$ and $$\mathrm{P}(\mathrm{D}-\alpha \mathrm{I})_{\mid_{M_{1}^{\prime} \tilde{\oplus} N}} = \mathrm{P}(\mathrm{D}-\alpha \mathrm{I})_{\mid_{M_{1}^{\prime} \tilde{\oplus} N}}.$$
\end{proof}
\begin{flushright}
	$\boxdot$
\end{flushright}
Similarly as for $ \mathcal{M}{\tilde{\Phi}_{0}}(H_{\mathcal{A}})  $ and $\widehat{\mathcal{M}\Phi}_{+}^{-}(H_{\mathcal{A}})  $, one can show that $\widehat{\mathcal{M}\Phi}_{-}^{+}(H_{\mathcal{A}})$ is open.
Indeed, let $D \in \widehat{\mathcal{M} \Phi}_{-}^{+}(H_{\mathcal{A}}),$ choose an $\widehat{\mathcal{M} \Phi}_{-}^{+} $ decomposition $H_{\mathcal{A}} = M_{1}^{\prime} \tilde \oplus {N_{1}^{\prime}}\stackrel{D}{\longrightarrow} M_{2}^{\prime} \tilde \oplus N_{2}^{\prime}= H_{\mathcal{A}}   $ for $D.$ Let $N$ be a closed submodule s.t. $H_{\mathcal{A}} =M_{1}^{\prime} \tilde \oplus N \tilde \oplus N_{2}^{\prime}   $ and s.t. the projection onto $M_{1}^{\prime} \tilde \oplus N  $ along $N_{2}^{\prime}   $ is adjointable. By the proof of  \cite[Lemma 2.7.10]{MT}, there exists an $\epsilon >0  $ s.t. if $ \parallel G-D \parallel < \epsilon $ for an operator $G \in B^{a}(H_{\mathcal{A}}).$ then G has the matrix
$\left\lbrack
\begin{array}{cc}
G_{1} & 0 \\
0 & G_{4} \\
\end{array}
\right \rbrack,
$
w.r.t. the decomposition $H_{\mathcal{A}} = M_{1}^{\prime} \tilde \oplus U({N_{1}^{\prime}})\stackrel{G}{\longrightarrow} V^{-1}(M_{2}^{\prime}) \tilde \oplus N_{2}^{\prime}= H_{\mathcal{A}},$ where $U, V, G_{1}$ are isomorphisms. It follows that $G \in \widehat{\mathcal{M} \Phi}_{-}^{+}(H_{\mathcal{A}}).$\\
If $\mathcal{A} =\mathbb{\mathrm{C}} $, that is if $H_{\mathcal{A}}=H$ is an ordinary Hilbert space, then $\mathcal{M}{\tilde{\Phi}_{0}}(H)={{\Phi}_{0}}(H)$, $\widehat{\mathcal{M}\Phi}_{+}^{-}(H)=\Phi_{+}^{-}(H) $ and $\widehat{\mathcal{M}\Phi}_{-}^{+}(H)=\Phi_{-}^{+}(H) .$
In addition, observe that $\widehat{\mathcal{M}\Phi}_{+}^{-}(H_{\mathcal{A}}) \subseteq {\mathcal{M}\Phi}_{+}^{-}(H_{\mathcal{A}}) $ and  $\widehat{\mathcal{M}\Phi}_{-}^{+}(H_{\mathcal{A}}) \subseteq {\mathcal{M}\Phi}_{-}^{+}(H_{\mathcal{A}}) .$
\section{The boundary of several kinds of Fredholm spectra in $Z(\mathcal{A})$}
Recall first 	\cite[Definition5.1]{I} and Definition \ref{d510} in Preliminaries. We give then the following definition:
\begin{definition} \label{d510} 
	Let $\mathrm{F} \in B^{a}(H_{\mathcal{A}}).$ We set
	$$\sigma_{ew}^{\mathcal{A}} (\mathrm{F})=\lbrace \alpha \in Z(\mathcal{A}) \mid (\mathrm{F}-\alpha \mathrm{I}) \notin \mathcal{M} \Phi_{0}(H_{\mathcal{A}}) \rbrace,$$
	$$\sigma_{e \alpha}^{\mathcal{A}} (\mathrm{F})=\lbrace \alpha \in Z(\mathcal{A}) \mid (\mathrm{F}-\alpha \mathrm{I}) \notin \mathcal{M} \Phi_{+}(H_{\mathcal{A}}) \rbrace,$$
	$$\sigma_{e \beta}^{\mathcal{A}} (\mathrm{F})=\lbrace \alpha \in Z(\mathcal{A}) \mid (\mathrm{F}-\alpha \mathrm{I}) \notin \mathcal{M} \Phi_{-}(H_{\mathcal{A}}) \rbrace,$$
	$$\sigma_{e k}^{\mathcal{A}} (\mathrm{F})=\lbrace \alpha \in Z(\mathcal{A}) \mid (\mathrm{F}-\alpha \mathrm{I}) \notin \mathcal{M} \Phi_{+}(H_{\mathcal{A}}) \cup \mathcal{M} \Phi_{-}(H_{\mathcal{A}}) \rbrace,$$
	$$\sigma_{e f}^{\mathcal{A}} (\mathrm{F})=\lbrace \alpha \in Z(\mathcal{A}) \mid (\mathrm{F}-\alpha \mathrm{I}) \notin \mathcal{M} \Phi(H_{\mathcal{A}}) \rbrace.$$ 
\end{definition}
\begin{theorem} \label{t520} 
	Let  $ \mathrm{F} \in B^{a}(H_{\mathcal{A}})   .$ Then the following inclusions hold:\\
	\begin{center}
		$ \partial \sigma_{ew}^{\mathcal{A}} (\mathrm{F}) \subseteq   \partial \sigma_{ef}^{\mathcal{A}} (\mathrm{F}) \subseteq $
		$	
		\begin{array}{l}
		{\partial \sigma_{e\beta}^{\mathcal{A}} (\mathrm{F})} \\
		{\partial \sigma_{e\alpha}^{\mathcal{A}} (\mathrm{F})} \\
		\end{array}
		$ $ \subseteq  \partial \sigma_{ek}^{\mathcal{A}} (\mathrm{F}) .$
		
	\end{center}
\end{theorem}
(We consider the boundaries in the $\mathrm{C}^{*}$-algebra $Z(\mathcal{A})$).\\
\begin{proof}
	We will show this by proving the following inclusions: 
	$$  \partial \sigma_{ew}^{\mathcal{A}} (\mathrm{F}) \subseteq   \sigma_{ef}^{\mathcal{A}} (\mathrm{F}),$$
	$$\partial \sigma_{ef}^{\mathcal{A}} (\mathrm{F}) \subseteq ( \sigma_{e\alpha}^{\mathcal{A}} (\mathrm{F}) \cap \sigma_{e\beta}^{\mathcal{A}} (\mathrm{F}))=\sigma_{ek}^{\mathcal{A}} (\mathrm{F}) ,$$
	$$ \partial \sigma_{e\alpha}^{\mathcal{A}} (\mathrm{F}) \subseteq  \sigma_{ek}^{\mathcal{A}} (\mathrm{F}) \textrm{ and } \partial \sigma_{e\beta}^{\mathcal{A}} (\mathrm{F}) \subseteq  \sigma_{ek}^{\mathcal{A}} (\mathrm{F}).$$
	Since obviously 
	\begin{center}
		$ \partial \sigma_{ek}^{\mathcal{A}} (\mathrm{F}) \subseteq   $
		$	
		\begin{array}{l}
		{ \sigma_{e\alpha}^{\mathcal{A}} (\mathrm{F})} \\
		{ \sigma_{e\beta}^{\mathcal{A}} (\mathrm{F})} \\
		\end{array}
		$ $ \subseteq \sigma_{ef}^{\mathcal{A}} (\mathrm{F}) \subseteq \sigma_{ew}^{\mathcal{A}} (\mathrm{F}) ,$
	\end{center}
	if we prove the inclusions above, the theorem would follow. Here we use the property that if $S,S^{\prime}  \subseteq  Z(\mathcal{A}) \textrm{ , } S \subseteq S^{\prime} \textrm{ and } \partial  S^{\prime} \subseteq S  \textrm{ , then }  \partial  S^{\prime} \subseteq \partial S.$ The first inclusion follows by the same arguments as in the classical case (the proof of \cite[Theorem 2.2.2.3]{ZZRD}) since  $ \sigma_{ew}^{\mathcal{A}} (\mathrm{F})\setminus   \sigma_{ef}^{\mathcal{A}} (\mathrm{F})    $ is open in  $ Z(\mathcal{A}) $  by the continuily of index, which follows from \cite[Lemma 2.7.10]{MT}. Next, if  $\alpha \in  \partial \sigma_{ef}^{\mathcal{A}} (\mathrm{F})  , $  then obviously  $\mathrm{F}-\alpha \mathrm{I}   $  is in  $ \partial \mathcal{M} \Phi (H_{\mathcal{A}})   .$ Using \cite[Corollary 4.2]{I} we deduce that  $(\mathrm{F}-\alpha \mathrm{I}) \notin  \mathcal{M} \Phi_{+}(H_{\mathcal{A}}) \cup \mathcal{M} \Phi_{-}(H_{\mathcal{A}})  .$ This works as in the proof of \cite[2.2.2.4]{ZZRD}  and \cite[2.2.2.5]{ZZRD}. Hence
	$$\partial \sigma_{ef}^{\mathcal{A}} (\mathrm{F}) \subseteq   ({ \sigma_{e\alpha}^{\mathcal{A}} (\mathrm{F})} \cap { \sigma_{e\beta}^{\mathcal{A}} (\mathrm{F})}) $$
	Suppose now that  $\tilde{\alpha} \in \partial \sigma_{e\alpha} (\mathrm{F})   .$ If  $\tilde{\alpha} \notin  \sigma_{e \beta }^{\mathcal{A}} (\mathrm{F}), $  then  $ (\mathrm{F}-\tilde \alpha \mathrm{I}) \in  \mathcal{M} \Phi_{-}(H_{\mathcal{A}})   ,$  so there exsists a decomposition  
	$$ H_{\mathcal{A}} = M_{1} \tilde \oplus {N_{1}}\stackrel{\mathrm{F}-\tilde{\alpha} \mathrm{I}}{\longrightarrow}  M_{2} \tilde \oplus N_{2}= H_{\mathcal{A}} $$  
	w.r.t. which  $\mathrm{F}-\tilde\alpha \mathrm{I}   $  has the matrix
	$\left\lbrack
	\begin{array}{cc}
	(\mathrm{F}-\tilde \alpha \mathrm{I})_{1} & 0 \\
	0 & (\mathrm{F}-\tilde \alpha \mathrm{I})_{4} \\
	\end{array}
	\right \rbrack
	,$ 
	where  $(\mathrm{F}-\tilde \alpha \mathrm{I})_{1},$  is an isomorphism and  $ N_{2}  $  is finitely generated. \\
	By the proof of \cite[Lemma 2.7.10]{MT} there exists some  $ \epsilon > 0   $  such that if  $ \tilde{\alpha} \in \mathcal{A}  $  and  $||\tilde \alpha - \tilde \alpha^{\prime} || < \epsilon  ,$  then  $\mathrm{F}-\tilde \alpha^{\prime} \mathrm{I}  $  has the matrix
	$\left\lbrack
	\begin{array}{cc}
	(\mathrm{F}-\tilde \alpha^{\prime} \mathrm{I})_{1} & 0 \\
	0 & (\mathrm{F}-\tilde \alpha^{\prime} \mathrm{I})_{4} \\
	\end{array}
	\right \rbrack
	$  
	w.r.t. the decomposition  
	$$ H_{\mathcal{A}} = M_{1} \tilde \oplus \mathrm{U}({N_{1}})\stackrel{\mathrm{F}- \tilde \alpha^{\prime} \mathrm{I}}{\longrightarrow} \mathrm{V}^{-1}(M_{2}) \tilde \oplus N_{2}= H_{\mathcal{A}} $$  
	where  $(\mathrm{F}-\tilde \alpha^{\prime} \mathrm{I})_{1}, \mathrm{U} ,\mathrm{V} $  are isomorphisms, so
	$(\mathrm{F}- \tilde \alpha^{\prime} \mathrm{I}) \in \mathcal{M} \Phi_{-}(H_{\mathcal{A}})  $ in this case. But, since  $\tilde \alpha \in  \partial \sigma_{e\alpha}^{\mathcal{A}} (\mathrm{F})   ,$ we may choose  $\tilde \alpha^{\prime} \in \mathcal{A}   $  such that  $||\tilde \alpha - \tilde \alpha^{\prime} || < \epsilon   $  and in addition  $(\mathrm{F}-\tilde{ \alpha^{\prime}}\mathrm{I}) \in \mathcal{M} \Phi_{+}(H_{\mathcal{A}})   .$ Thus  $(\mathrm{F}-\tilde \alpha^{\prime} \mathrm{I}) \in \mathcal{M} \Phi_{+}(H_{\mathcal{A}}) \cap  \mathcal{M} \Phi_{-}(H_{\mathcal{A}})  $  and from \cite[Corollary 2.4]{I}, we have  $\mathcal{M} \Phi_{+}(H_{\mathcal{A}}) \cap  \mathcal{M} \Phi_{-}(H_{\mathcal{A}}) = \mathcal{M} \Phi(H_{\mathcal{A}})    ,$ so \\
	$(\mathrm{F}-\tilde \alpha^{\prime} \mathrm{I}) \in \mathcal{M} \Phi(H_{\mathcal{A}}) .$ Since  $\mathrm{F}-\tilde \alpha^{\prime} \mathrm{I}   $  has the matrix 
	$\left\lbrack
	\begin{array}{cc}
	(\mathrm{F}-\tilde \alpha^{\prime} \mathrm{I})_{1} & 0 \\
	0 & (\mathrm{F}-\tilde \alpha^{\prime} \mathrm{I})_{4} \\
	\end{array}
	\right \rbrack
	$  
	w.r.t. the decomposition
	$$ H_{\mathcal{A}} = M_{1} \tilde \oplus \mathrm{U}({N_{1}})\stackrel{\mathrm{F}- \tilde \alpha^{\prime} \mathrm{I}}{\longrightarrow} \mathrm{V}^{-1}(M_{2}) \tilde \oplus N_{2}= H_{\mathcal{A}} ,$$   
	where  $(\mathrm{F}-\tilde \alpha^{\prime} \mathrm{I})_{1}  ,$  $\mathrm{V}$ are isomorphisms and  $ N_{2}  $  is finitely generated, by\\
	\cite[Lemma 2.17]{I} we must have that  $ \mathrm{U}({N_{1}})  $  is finitely generated, as\\
	$(\mathrm{F}-\tilde \alpha^{\prime} \mathrm{I}) \in \mathcal{M} \Phi(H_{\mathcal{A}}).$ Hence  $N_{1}   $  is finitely generated, so
	$(\mathrm{F}-\tilde \alpha \mathrm{I}) \in \mathcal{M} \Phi(H_{\mathcal{A}})    .$ In particular  $ (\mathrm{F}-\tilde \alpha \mathrm{I}) \in \mathcal{M} \Phi_{+}(H_{\mathcal{A}}),$  which contradicts the choice of  $\tilde \alpha \in \partial \sigma_{e \alpha}^{\mathcal{A}} (\mathrm{F})   .$ Thus  $\tilde \alpha \in  \sigma_{e \alpha}^{\mathcal{A}} (\mathrm{F})    ,$ so  $\tilde \alpha \in \sigma_{e \alpha}^{\mathcal{A}} (\mathrm{F}) \cap \sigma_{e \beta}^{\mathcal{A}} (\mathrm{F}) = \sigma_{e k}^{\mathcal{A}} (\mathrm{F}).    $
	Similarly, we can show that  $ \partial \sigma_{e \beta }^{\mathcal{A}} (\mathrm{F}) \subseteq \sigma_{e k}^{\mathcal{A}} (\mathrm{F}) .$
\end{proof} 
\begin{flushright}
	$\boxdot$
\end{flushright}
Next we consider the following spectra for  $\mathrm{F}\in B^{a}(H_{\mathcal{A}})   :$
$$\sigma_{e\tilde a}^{\mathcal{A}} (\mathrm{F})=\lbrace \alpha \in Z(\mathcal{A}) \mid (\mathrm{F}-\alpha \mathrm{I}) \notin \tilde{\mathcal{M} \Phi}_{+}^{-}(H_{\mathcal{A}}) \rbrace$$
$$\sigma_{ea}^{\mathcal{A}} (\mathrm{F})=\lbrace  \alpha \in Z(\mathcal{A}) \mid (\mathrm{F}-\alpha \mathrm{I}) \notin \mathcal{M} \Phi_{+}^{-}(H_{\mathcal{A}}) \rbrace $$
Clearly,  $\sigma_{ea}^{\mathcal{A}} (\mathrm{F}) \subseteq \sigma_{e a^{\prime}}^{\mathcal{A}} (\mathrm{F}) \subseteq \sigma_{e\tilde a}^{\mathcal{A}} (\mathrm{F}) $ . We have the following theorem.
\begin{theorem} \label{T530}
	Let  $\mathrm{F} \in B^{a}(H_{\mathcal{A}})   .$ Then  
	$$ \partial \sigma_{ew}^{\mathcal{A}} (\mathrm{F}) \subseteq  \partial \sigma_{e\tilde{a}}^{\mathcal{A}} (\mathrm{F}) \subseteq   \partial \sigma_{e{a}}^{\mathcal{A}} (\mathrm{F})$$ 
	Moreover, $ \partial \sigma_{e{a}}^{\mathcal{A}} (\mathrm{F}) \subseteq   \partial \sigma_{e\alpha}^{\mathcal{A}} (\mathrm{F}) $ if $K({\mathcal{A}})$ satisfies the cancellation property.
\end{theorem}
\begin{proof}
	Again it suffices to show  
	$$\partial \sigma_{ew}^{\mathcal{A}} (\mathrm{F}) \subseteq 
	\sigma_{e\tilde{a}}^{\mathcal{A}} (\mathrm{F}), 
	\partial \sigma_{e\tilde{a}}^{\mathcal{A}} (\mathrm{F})  \subseteq 
	\sigma_{e{a}}^{\mathcal{A}} (\mathrm{F})   \textrm{ and } 
	\partial \sigma_{e{a}}^{\mathcal{A}} (\mathrm{F}) \subseteq 
	\sigma_{e\alpha}^{\mathcal{A}} (\mathrm{F})    .$$
	The first inclusion follows as in the proof of \cite[Theorem 2.7.5]{ZZRD},  since\\
	$\partial \sigma_{ew}^{\mathcal{A}} (\mathrm{F}) \subseteq \partial \sigma_{ek}^{\mathcal{A}} (\mathrm{F})  $  by Theorem \ref{t520} and since  $ \partial \sigma_{ek}^{\mathcal{A}} (\mathrm{F}) \subseteq  \sigma_{e\tilde{a}}^{\mathcal{A}} (\mathrm{F})  .$ \\
	To deduce the second inclusion, assume first that  $\alpha \in \partial \sigma_{e\tilde{a}}^{\mathcal{A}} (\mathrm{F}) \setminus    \sigma_{e{a}}^{\mathcal{A}} (\mathrm{F})  .$ Then  $(\mathrm{F}- \alpha \mathrm{I}) \in \mathcal{M} \Phi_{+}^{-}(H_{\mathcal{A}})   $  
	and  $ (\mathrm{F}- \alpha \mathrm{I}) \notin \tilde{\mathcal{M}} \Phi_{+}^{-}(H_{\mathcal{A}})  .$ It follows then that  $\mathrm{F}- \alpha \mathrm{I}   $  is in $ \mathcal{M} \Phi_{+}(H_{\mathcal{A}}) \setminus \mathcal{M} \Phi (H_{\mathcal{A}})   .$ But, since  $\mathcal{M} \Phi_{+}(H_{\mathcal{A}}) \setminus \mathcal{M} \Phi (H_{\mathcal{A}})    $  is open 
	by \cite[Theorem 4.1]{I} and  $\tilde{\mathcal{M}} \Phi_{+}^{-}(H_{\mathcal{A}})  \subseteq  \mathcal{M} \Phi (H_{\mathcal{A}})     $  by definition, it follows that  $ (\mathrm{F}- \alpha \mathrm{I}) \notin  \partial \tilde{\mathcal{M}} \Phi_{+}^{-}(H_{\mathcal{A}})   .$ This contradicts the choice of  $ \alpha \in \partial \sigma_{e\tilde{a}}^{\mathcal{A}} (\mathrm{F})  .$\\
	Hence  $\partial \sigma_{e\tilde{a}}^{\mathcal{A}} (\mathrm{F})  \subseteq  \sigma_{e{a}}^{\mathcal{A}} (\mathrm{F}) .$ For the last inclusion, assume that  $ \tilde{\alpha} \in \partial \sigma_{e{a}}^{\mathcal{A}} (\mathrm{F})    $  and that  $\tilde{\alpha} \notin  \sigma_{e\alpha}^{\mathcal{A}} (\mathrm{F})   .$
	Then  $(\mathrm{F}- \tilde{\alpha} \mathrm{I}) \in \mathcal{M} \Phi_{+}(H_{\mathcal{A}})   $  and  $(\mathrm{F}- \tilde{\alpha} \mathrm{I}) \notin \mathcal{M} \Phi_{+}^{-}(H_{\mathcal{A}})    .$ This means, by the definitions of  $\mathcal{M} \Phi_{+}^{-}(H_{\mathcal{A}})   $ and  $\mathcal{M} \Phi_{+}(H_{\mathcal{A}})    $ that $(\mathrm{F}- \tilde{\alpha} \mathrm{I}) \in \mathcal{M} \Phi(H_{\mathcal{A}}) $   and that given any decomposition  
	$$  H_{\mathcal{A}} = M_{1} \tilde \oplus {N_{1}}\stackrel{\mathrm{F}- \tilde{\alpha} \mathrm{I}}{\longrightarrow}   M_{2} \tilde \oplus N_{2}= H_{\mathcal{A}}  $$  
	w.r.t. which  $ (\mathrm{F}- \tilde{\alpha} \mathrm{I})  $  has the matrix
	$\left\lbrack
	\begin{array}{cc}
	(\mathrm{F}- \alpha \mathrm{I})_{1} & 0 \\
	0 & (\mathrm{F}- \tilde{\alpha} \mathrm{I})_{4} \\
	\end{array}
	\right \rbrack
	,$   
	where  $(\mathrm{F}- \tilde{\alpha} \mathrm{I})_{1}   $  is an isomorphism and  $ N_{1},N_{2}  $  are finitely generated, then  $ N_{1}  $  is\underline{ not} isomorphic to a closed submodule of  $N_{2}   $ . By the proof of  \cite[Lemma 2.7.10]{MT}  there exists an  $ \epsilon >0  $  such that if  $ \tilde{\alpha^{\prime}} \in \mathcal{A} $  and  $||\alpha - \tilde{\alpha}^{\prime}||< \epsilon , $  then  $(\mathrm{F}-  \tilde \alpha^{\prime} \mathrm{I}) \in  \mathcal{M} \Phi (H_{\mathcal{A}})  $  and  $(\mathrm{F}-  \tilde \alpha^{\prime} \mathrm{I})   $  has the matrix
	$\left\lbrack
	\begin{array}{cc}
	(\mathrm{F}-\tilde \alpha^{\prime} \mathrm{I})_{1} & 0 \\
	0 & (\mathrm{F}-\tilde \alpha^{\prime} \mathrm{I})_{4} \\
	\end{array}
	\right \rbrack
	$   
	w.r.t. the decomposition  
	$$  H_{\mathcal{A}} = M_{1} \tilde \oplus \mathrm{U}({N_{1}})\stackrel{\mathrm{F}- \tilde \alpha^{\prime} \mathrm{I}}{\longrightarrow} \mathrm{V}^{-1} (M_{2}) \tilde \oplus N_{2}= H_{\mathcal{A}} , $$  where  $(\mathrm{F}-\tilde \alpha^{\prime} \mathrm{I})_{1}   ,\mathrm{U},\mathrm{V}$  are isomorphisms. As  $ N_{1}  $  is not isomorphic to a closed submodule od  $N_{2}    $  and  $ \mathrm{U}  $  is an isomorphism from  $  H_{\mathcal{A}}  $  onto  $  H_{\mathcal{A}}  $ , it follows that  $ \mathrm{U}({N_{1}})  $  is \underline{not} isomorphic to a closed submodule of  $ N_{2}   $ . Now, if\\
	$(\mathrm{F}-\tilde{\alpha}^{\prime} \mathrm{I}) \in \mathcal{M} \Phi_{+}^{-}(H_{\mathcal{A}})   ,$  then we must have  $(\mathrm{F}-\tilde{ \alpha}^{\prime} \mathrm{I}) \in \tilde{\mathcal{M}}\Phi_{+}^{-}(H_{\mathcal{A}})    ,$\\
	as  $(\mathrm{F}-\tilde{\alpha}^{\prime} \mathrm{I}) \in  \mathcal{M} \Phi(H_{\mathcal{A}})   $  and  $ \tilde{\mathcal{M}} \Phi_{+}^{-}(H_{\mathcal{A}}) = \mathcal{M} \Phi_{+}^{-}(H_{\mathcal{A}})  \cap \mathcal{M} \Phi(H_{\mathcal{A}})    $  by definition. By \cite[Lemma 5.2]{I}, as $ K(\mathcal{A})$ satisfies "the cancelation" property, we must then have that  $ \mathrm{U}({N_{1}}) \preceq N_{2}   $  which is a contradiction. So  $ \partial \sigma_{e{a}}^{\mathcal{A}} (\mathrm{F}) \subseteq  \sigma_{e\alpha} (\mathrm{F}) .$
\end{proof}
\begin{flushright}
	$\boxdot$
\end{flushright}
Similarly one can show that
$$ \partial \sigma_{e w}^{\mathcal{A}} (\mathrm{F})  \subseteq  \partial \sigma_{e \tilde b}^{\mathcal{A}} (\mathrm{F}) \subseteq  \partial \sigma_{eb}^{\mathcal{A}} (\mathrm{F})  $$  
where  $$ \sigma_{e \tilde b}^{\mathcal{A}} (\mathrm{F})=\lbrace \alpha \in \mathcal{A}| (\mathrm{F}-\alpha \mathrm{I}) \notin \mathcal{M} \tilde{\Phi}_{-}^{+}(H_{\mathcal{A}})  \rbrace $$  
and
$$\sigma_{e b}^{\mathcal{A}} (\mathrm{F})=\lbrace \alpha \in \mathcal{A}| (\mathrm{F}-\alpha \mathrm{I}) \notin \mathcal{M} {\Phi}_{-}^{+}(H_{\mathcal{A}})   \rbrace .$$ 
and in addition
$\partial \sigma_{e b}^{\mathcal{A}} (\mathrm{F})  \subseteq  \partial \sigma_{e \beta}^{\mathcal{A}} (\mathrm{F})$ if $K({\mathcal{A}})$ satisfies "the cancellation property".\\

%\begin{acknowledgement}

\begin{flushleft}
\textbf{Acknowledgement } I am especially grateful to my supervisor Professor Vladimir M. Manuilov for careful reading of my paper and for inspiring comments that led to the improved presentation of the paper. Also I am grateful to Professor Dragan S. Djordjevic for suggesting the research topic of this paper and for introducing to me the relevant reference books. Finally, I am grateful to the Referee for the comments and suggestions that led to the improved presentation of the paper.
\end{flushleft}
%\end{acknowledgement}

\begin{flushleft}
	{Stefan Ivkovi\'{c}}\\
	The Mathematical Institute of the Serbian Academy of Sciences and Arts, \\
	p.p. 367, Kneza Mihaila 36, 11000 Beograd, Serbia,\\
	Tel.: +381-69-774237 \\
	\email{stefan.iv10@outlook.com}     
\end{flushleft}      %  \\	
	
\end{document}